\documentclass[11pt, psamsfonts, reqno]{amsart}
\usepackage{amsmath}
\usepackage{amsthm}
\usepackage{amssymb}
\usepackage{amscd}
\usepackage{amsfonts}
\usepackage{amsbsy}
\usepackage{epsfig} 
\usepackage[bookmarks=false]{hyperref}
\usepackage{mathrsfs}
\usepackage{pdfsync}
\usepackage{enumitem}

\newtheorem{theorem}{Theorem}
\newtheorem{remark}{Remark}
\newtheorem{proposition}{Proposition}
\newtheorem{lemma}{Lemma}

\newtheorem{definition}{Definition}

\newtheorem{maintheorem}{Theorem}

\def\ie{{\em i.e.,\ }}

\def\cf{{\em cf.\ }}

\newfont\bbf{msbm10 at 12pt}
\def\eps{\varepsilon}
\def\ph{\varphi}
\def\R{{\mathbb R}}

\def\N{{\mathbb N}}

\def\E{{\mathbb E}}

\def\M{{\mathcal M}}

\def\hot{{\hbox{{\rm h.o.t.}}}}

\def\feig{f_{\mbox{\tiny feig}}}

\def\Bas{\mbox{Bas}}
\def\w{\tilde w^t}

\def\Fib{\mbox{Fib}_\ell}
\def\FP{\Theta}
\def\le{\leqslant}
\def\ge{\geqslant}
\def\htop{h_{top}}
\def\Pconf{ P_{\mbox{\rm\tiny Conf}} }

\def\drift{\mbox{\bf \it Dr}}

\newcommand{\st}{such that }

\def\equi{\Leftrightarrow}
\newdimen\AAdi%
\newbox\AAbo%
\def\AArm{\fam0 }%\tenrm}%
\def\AAk#1#2{\setbox\AAbo=\hbox{#2}\AAdi=\wd\AAbo\kern#1\AAdi{}}%
\def\AAr#1#2#3{\setbox\AAbo=\hbox{#2}\AAdi=\ht\AAbo\raise#1\AAdi\hbox{#3}}%
\def\1{{\AArm 1\AAk{-.8}{I}I}}%

\begin{document}

\title[Wild attractors and thermodynamic formalism]
{Wild attractors and thermodynamic formalism.
%\\[5mm]Attracteurs de Cantor et Formalisme Thermodynamique
}

  	%One-parameter continuous families of measure-preserving transformations


\author{Henk Bruin}\address{Faculty of Mathematics, University of Vienna, 
Oskar Morgensternplatz 1, 1090 Vienna, Austria}
\email{henk.bruin@univie.ac.at}
\urladdr{\url{http://www.mat.univie.ac.at/~bruin}}

\author{Mike Todd}\address{
Mathematical Institute,
University of St Andrews,
North Haugh,
St Andrews,
Fife,
KY16 9SS,
Scotland}
\email{m.todd@st-andrews.ac.uk}
\urladdr{\url{http://www.mcs.st-and.ac.uk/~miket/index.html}}

\date{Version of \today}
\thanks{
The hospitality of the Mittag-Leffler Institute
in Stockholm
(2010 Spring programme on Dynamics and PDEs) is gratefully acknowledged.
Parts of this paper are based on notes written in 1994-1995
when HB had a research fellowship at the University of Erlangen-Nuremberg, 
funded by the Netherlands Organisation for Scientific Research (NWO). MT was partially supported by NSF grants DMS 0606343 and DMS 0908093.
}

\begin{abstract}
Fibonacci unimodal maps can have a wild Cantor attractor, 
and hence be Lebesgue dissipative, depending on the order of the critical point.
We present a one-parameter family $f_\lambda$
of countably piecewise linear unimodal Fibonacci maps 
in order to study
the thermodynamic formalism of dynamics where dissipativity of
Lebesgue (and conformal) measure is responsible for phase transitions.
We show that for the potential $\phi_t = -t\log|f'_\lambda|$,
there is a unique phase transition at some $t_1 \le 1$, and
the pressure $P(\phi_t)$ is analytic (with unique equilibrium state)
elsewhere. The pressure is majorised by a non-analytic 
$C^\infty$ curve (with all derivatives equal to $0$ at $t_1 < 1$) at the emergence of
a wild attractor, whereas the phase transition at $t_1 = 1$ can be of any finite order
for those $\lambda$ for which $f_\lambda$ is Lebesgue conservative.
We also obtain results on the existence of conformal measures and equilibrium states, as well as the hyperbolic dimension and the dimension of the basin of $\omega(c)$.
\iffalse
\\[5mm]
{\sc R\'esum\'e.}  Les applications unimodales de Fibonacci peuvent avoir des attracteurs de Cantor, selon  l'ordre du point critique, et dans ce cas, la mesure de Lebesgue est dissipative.
Nous pr\'esentons ici une famille \`a un param\`etre  d'applications unimodales de Fibonacci, $f_\lambda$. Elles sont toutes affines par morceaux  avec un ensemble d\'enombrable de morceaux. 
Le but est d'\'etudier le formalisme thermodynamique lorsque la dissipativit\'e de la mesure de Lebesgue
g\'en\`ere des transitions de phase.
Nous montrons que pour le potentiel $\phi_t = -t \log |f'_\lambda|$,
il y a une transition de phase unique pour un certain $t_1 \le 1$, la fonction pression $t\mapsto P(\phi_t)$ \'etant analytique (avec une mesure d'\'equilibre unique) partout ailleurs.
\`A l'apparition de l'attracteur, le graphe de pression est 
 bord\'e par une courbe non-analytique (de classe) $C^\infty$ (avec toutes ses d\'eriv\'ees nulles en $t_1$)
% mais, pour les param\`etres pour lesquels la mesure de Lebesgue est conservative, 
alors que pour des valeurs du param\`etre pour lesquelles l'application \`a une mesure de Lebesgue conservative,
la transition de phase peut avoir n'importe quel ordre fini.
Nous obtenons aussi des r\'esultats sur l'existence de mesures conformes et d'\'etats d'\'equilibre, ainsi que sur la dimension hyperbolique et la dimension du bassin de $\omega(c)$.
\fi
\end{abstract}

\maketitle

\section{Introduction}\label{sec:intro}

The aim of this paper is to understand thermodynamic formalism
of unimodal interval maps $f:I \to I$ on the boundary between 
conservative and dissipative behaviour.
For a `geometric' potential $\phi_t = -t \log|f'|$, the {\em pressure function}
is defined by
\begin{equation}\label{eq:pressure}
P(\phi_t) = \sup\left\{ h_\mu + \int \phi_t \ d\mu : \mu \in \M,
\int \phi_t~ d\mu > -\infty \right\},
\end{equation}
where the supremum is taken over the set $\M$ of $f$-invariant probability
measures $\mu$, and $h_\mu$ denotes the entropy of the measure. 
A measure $\mu_t \in \M$ that assumes this supremum is called an 
{\em equilibrium state}.
Pressure is a convex and non-increasing function in $t$ 
and $P(\phi_0) = h_{top}(f)$ is the topological entropy of $f$.
At most parameters $t$, the pressure function $t \mapsto P(\phi_t)$ is analytic,
and there is a unique equilibrium state which depends continuously on $t$. 
If the pressure function fails
to be analytic at some $t$, then we speak of a \emph{phase transition} at $t$, which 
hints at a qualitative (and discontinuous), rather than quantitative, 
change in equilibrium states. 
Refining this further, if the pressure function is $C^{n-1}$ at $t$, but not $C^n$, we say that there is an \emph{$n$-th order phase transition} at $t$.

Given a unimodal map $f$ with critical point $c$, we say that the critical point is
 \emph{non-flat}  if there exists a diffeomorphism $\phi:\R \to
\R$ with $\phi(0)=0$ and $1<\ell<\infty$ \st for $x$ close to
$c$, $f(x)=f(c)\pm|\phi(x-c)|^{\ell}$.  The value of
$\ell = \ell_c$ is known as the \emph{critical order} of $c$.  The metric behaviour of a unimodal map is essentially determined by its topological/combinatorial properties plus its critical order.
We give a brief summary of what is known for $C^2$ unimodal maps
with non-flat critical point. 
A first result is due to Ledrappier \cite{Ledrap} who proved that a measure $\mu \in \M$
of positive entropy is an equilibrium state for $t=1$ if and only if
$\mu$ is absolutely continuous w.r.t.\ Lebesgue (abbreviate acip).
This also shows that $t=1$ is the expected first zero of the pressure function.
For simplicity, we assume in the
classification below that $f$ is 
topologically transitive on its dynamical core $[f^2(c), f(c)]$, \ie there exists a point  $x_0$ such that $\overline{\cup_{n\ge 0}f^n(x_0)}=[f^2(c), f(c)]$),
except in cases (1) and (5).

\begin{enumerate}
\item If the critical point $c$ of $f$ is attracted
to an attracting periodic orbit, then the non-wandering set is
hyperbolic on which Bowen's theory \cite{Bowen} applies in its entirety.
In particular, no phase transitions occur.
\item If $f$ satisfies the
Collet-Eckmann condition, \ie derivatives along the critical orbit
grow exponentially fast, then the pressure
 is analytic in a neighbourhood of $t = 1$, \cite{BK}; and $C^1$ for all $t<1$ except when the critical point is preperiodic, \cite{ITeqnat}.  
An example of the preperiodic critical point case is the Chebyshev polynomial $x\mapsto 4x(1-x)$  which, 
as in the much more general work of 
Makarov \& Smirnov \cite{MakSmi}, has a phase transition at $t=-1$.
The pressure function is affine for $t \neq -1$ in this case.
\footnote{Collet-Eckmann maps with ``low-temperature phase transitions'' were found in 
\cite{CLR},  after our paper was first submitted, but which we can include in this revision.}

\item If $f$ is non-Collet-Eckmann but possesses an acip $\mu_{ac}$, then 
there is a first order phase transition at $t=1$ 
(\ie $t\mapsto P(\phi_t)$ is continuous but not $C^1$).
More precisely, $P(\phi_t) = 0$ if and only if $t \ge 1$ and
the left derivative $\lim_{s \uparrow 1} \frac{d}{ds} P(\phi_s) = -\lambda(\mu_{ac}) < 0$, 
where $\lambda(\mu_{ac})=\int\log|f'|~d\mu_{ac}$ denotes the \emph{Lyapunov exponent} 
of $\mu_{ac}$, see \cite[Proposition 1.2]{ITeqnat}. 
\item If $f$ is non-Collet-Eckmann but has an absolutely continuous conservative 
infinite $\sigma$-finite measure, then there is still a phase transition at $t=1$, but $P(\phi_t)$ is $C^1$. In fact, 
$P(\phi_t) = 0$ if and only if $t \ge 1$ and
the left derivative $\lim_{s \uparrow 1} \frac{d}{ds} P(\phi_s) = 0$.
This follows from the proof of \cite[Lemma 9.2]{ITeqnat}. 
\item If $f$ is infinitely renormalisable, then the critical omega-limit set $\omega(c)$ is a Lyapunov stable attractor, and its basin
$\Bas = \{ x : \omega(x) \subset \omega(c) \}$ 
is a second Baire category set of full Lebesgue measure.
The best known example is the Feigenbaum-Coullet-Tresser map
$\feig$, for which the topological entropy $\htop(\feig) = 0$, and so
$P(\phi_t) \equiv 0$ for all $t \ge 0$.
More complicated renormalisation patterns can lead to a more 
interesting thermodynamic behaviour, see Avila \& Lyubich 
\cite{AvLyu_feig}, 
Moreira \& Smania \cite{MorSma} and Dobbs \cite{Dobbs}.
However, this thermodynamic behaviour is primarily a topological, rather than a metric,
phenomenon, so should be seen as complementary to the results given in
this paper.
\item If $f$ has a wild attractor, then $\omega(c)$ is not
Lyapunov stable and attracts a set of full Lebesgue measure, whereas
a second Baire category set of points has a dense orbit in $[f^2(c), f(c)]$.
In \cite[Theorem 10.5]{AvLyu_feig} it is asserted that there exists some $t_1<1$ such that 
 $P(\phi_t) = 0$ for $t \ge t_1$.  In this paper we study this fact, as well as further thermodynamic properties of wild attractors, in detail.
\end{enumerate}
A wild attractor occurs for a unimodal map $f$ if it has very large
critical order $\ell$ as well as Fibonacci combinatorics,
\ie the cutting times are the Fibonacci numbers.
(The {\em cutting times} $(S_k)_{k \ge 0}$ are the sequence of
iterates $n$ at which the image of the 
central branch of $f^n$ contains the critical point.
They satisfy the recursive formula $S_k - S_{k-1} = S_{Q(k)}$
for the so-called {\em kneading map} $Q : \N \to \N_0$;
so Fibonacci maps have kneading map $Q(k) = \max\{k-2, 0\}$,
see Section~\ref{sec:cpl} for more precise details.)

Let us parametrise Fibonacci maps by critical order, say
$$
\Fib:[0,1] \to [0,1], \qquad x  \mapsto a(\ell)(1- |2x-1|^{\ell}),
$$
where $a(\ell) \in [0,1]$ is chosen such that $\Fib$ has Fibonacci combinatorics. The picture is then as follows:
\[
\left\{ \begin{array}{ll}
\ell \le 2 & \Fib  \text{ has an acip which is super-polynomially mixing, \cite{LM, BLS},} \\[2mm]
2 < \ell < 2+\eps & \Fib  \text{ has an acip which is polynomially mixing
with exponent}\\
& \quad \text{ tending to infinity as $\ell \to 2$, \cite{KN, RLS}, } \\[2mm]
\ell_0 < \ell < \ell_1 & \Fib  \text{ has a conservative $\sigma$-finite acim, } \\[2mm]
\ell_1 < \ell  & \Fib  \text{ has a wild attractor \cite{BKNS},
with dissipative $\sigma$-finite}\\
& \quad \text{ acim, \cite{Mar}.}
\end{array}\right.
\]
For the logistic family (\ie critical order is 2), Lyubich proved there cannot be a
wild attractor, so in particular $\Fib$ has no wild attractor. In \cite{KN} 
it was shown that $\ell = 2+\eps$ still
does not allow for a wild attractor for $\Fib$. Wild attractors were shown to
exist \cite{BKNS}  for very large $\ell$.
The value of $\ell_1$ beyond which the existence of a wild attractor
is rigorously proven in \cite{BKNS} is extremely large\footnote{For less restrictive Fibonacci-like combinatorics
(basically if $k-Q(k)$ is bounded) the existence of
wild attractors was proved in \cite{Btams}.}, but
unpublished numerical simulations by Sutherland et al.\
suggest that $\ell_1 = 8$ suffices.
The region $\ell \in (\ell_0, \ell_1)$ is somewhat hypothetical. 
It can be shown \cite{Bthesis} that 
$\Fib$ has an absolutely continuous $\sigma$-finite measure
for $\ell > \ell_0$, 
and it stands to reason that this happens
before $\Fib$ becomes Lebesgue dissipative, but we have
no proof that indeed
$\ell_0 < \ell_1$, nor that this behaviour occurs on exactly a single interval.
The existence of a dissipative  $\sigma$-finite acim when there is a wild attractor was shown by Martens
\cite{Mar}, see also \cite[Theorem 3.1]{BH}.

Within interval dynamics, inducing schemes have become a standard
tool to study thermodynamic formalism, \cite{BT1, BT2, PS, Sartherm, BarIom_HS}.
One constructs a {\em full-branched} Gibbs-Markov induced
system $(Y, F)$  whose thermodynamic properties can be understood in terms
of a full shift on a countable alphabet.
However, precisely in the setting of wild attractors,
the set
\[
Y^\infty = \{ y \in Y : F^n(y) \text{ is well-defined for all } n \ge 0\}
\]
is dense in $Y$ but of zero Lebesgue measure $m$.
For this reason, we prefer to work with a different induced system,
called $(Y,F)$ again, that has branches of arbitrarily short
length, but for which $Y^\infty$ is co-countable.
By viewing the dynamics under $F$ as a random walk, we can show that
transience\footnote{We discuss transience and (null and positive) recurrence in detail in Section~\ref{sec:CMS}.} 
of this random walk (w.r.t.\ Lebesgue measure) implies
the existence of a Cantor attractor.

Proving transience of $(Y, F , m)$  is very technical due to the severe 
non-linearity of $F$ for smooth unimodal maps $f$ with large critical order.
For this reason, we introduce countably piecewise linear unimodal
maps for which induced systems with linear branches can be constructed.
This idea is definitely not new, \cf the maps of
Gaspard \& Wang \cite{GW} and L\"uroth \cite{Luroth, DajKra_book} 
as countably piecewise linear versions of the Farey and Gauss map, respectively.\footnote{In fact, considering $-t \log |f'|$ for the Gaspard \& Wang map is exactly
equivalent to the Hofbauer potential \cite{Hof} for the full shift on two symbols.}
The explicit construction for unimodal maps is new, however.
Although we are mostly interested in Fibonacci
maps, the method works in far more generality; it definitely suffices if
the kneading map $Q(k) \to  \infty$ and a technical condition \eqref{12.1}
is satisfied.
Note that the inducing scheme we will use
is somewhat different from that in \cite{BKNS} which was based on preimages of the fixed point.
Instead, we will use an inducing scheme based on precritical points,
used before in \cite{Bthesis}, and we arrive
at a two-to-one cover of a countably piecewise interval map 
$T_\lambda:(0,1] \to (0,1]$ defined in Stratmann \& Vogt \cite{StraVogt} 
as follows:
For $n \ge 1$, let $V_n:=(\lambda^{n}, \lambda^{n-1}]$ and define
\begin{equation}
T_\lambda(x):= \begin{cases}\ 
\frac{x-\lambda}{1-\lambda} & \text{ if } x\in V_1,\\[2mm]
\ \frac{x-\lambda^{n}}{\lambda(1-\lambda)} & \text{ if } x\in V_n, \quad n \ge 2.
\end{cases}
\label{eq:StrVogt}
\end{equation}

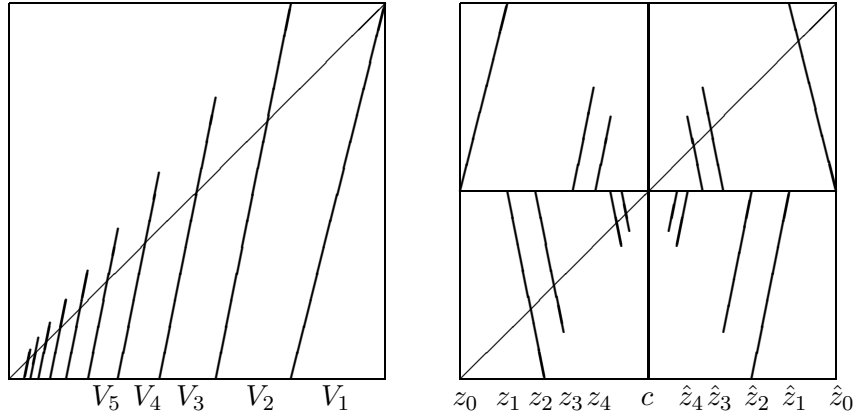
\begin{figure}[ht]
\unitlength=5mm
\begin{picture}(22,12)(0,-0.5) \let\ts\textstyle
\thinlines
\put(0,0){\line(1,0){10}}\put(0,10){\line(1,0){10}}
\put(0,0){\line(0,1){10}} \put(10,0){\line(0,1){10}}
\put(0,0){\line(1,1){10}}
\thicklines
\put(7.5,0){\line(1,4){2.5}} \put(8.3,-0.7){$\tiny V_1$}
\put(5.5,0){\line(1,5){2}} \put(6.3,-0.7){$\tiny V_2$}
\put(4,0){\line(1,5){1.5}}  \put(4.45,-0.7){$\tiny V_3$}
\put(2.9,0){\line(1,5){1.1}}  \put(3.3,-0.7){$\tiny V_4$}
\put(2.1,0){\line(1,5){0.8}}  \put(2.2,-0.7){$\tiny V_5$}
\put(1.52,0){\line(1,5){0.58}}
\put(1.095,0){\line(1,5){0.425}}
\put(0.791,0){\line(1,5){0.304}}
\put(0.572,0){\line(1,5){0.219}}
\put(0.4138,0){\line(1,5){0.1582}}
\put(0.2974,0){\line(1,5){0.1164}}
\put(0.21464,0){\line(1,5){0.08276}}
\thinlines
\put(12,0){\line(1,0){10}}\put(12,10){\line(1,0){10}}
\put(12,0){\line(0,1){10}} \put(22,0){\line(0,1){10}}
\put(17,0){\line(0,1){10}} \put(12,5){\line(1,0){10}}
\put(12,0){\line(1,1){10}}
\thicklines
\put(20.75,10){\line(1,-4){1.25}} \put(21.8,-0.7){$\tiny \hat z_0$}
\put(12,5){\line(1,4){1.25}} \put(11.8,-0.7){$\tiny z_0$}
\put(19.75,0){\line(1,5){1}} \put(20.55,-0.7){$\tiny \hat z_1$}
\put(13.25,5){\line(1,-5){1}} \put(12.95,-0.7){$\tiny z_1$}
\put(19,1.25){\line(1,5){0.75}} \put(19.55,-0.7){$\tiny \hat z_2$}
\put(14,5){\line(1,-5){0.75}} \put(13.8,-0.7){$\tiny z_2$}
\put(19,5){\line(-1,5){0.55}} \put(18.55,-0.7){$\tiny \hat z_3$}
\put(15,05){\line(1,5){0.55}} \put(14.6,-0.7){$\tiny z_3$}
\put(18.45,5){\line(-1,5){0.4}} \put(17.8,-0.7){$\tiny \hat z_4$}
\put(15.6,5){\line(1,5){0.4}} \put(15.35,-0.7){$\tiny z_4$}
\put(18.05,5){\line(-1,-5){0.29}}
\put(16,5){\line(1,-5){0.29}}
\put(17.76,5){\line(-1,-5){0.2125}}
\put(16.29,5){\line(1,-5){0.2125}}  \put(16.8,-0.7){$\tiny c$}
\end{picture}
\caption{The maps $T_\lambda:[0,1] \to [0,1]$ and $F_\lambda:[z_0, \hat z_0]
  \to
[z_0, \hat z_0]$. }
\label{fig1}
\end{figure}

In Section~\ref{sec:cpl}, we will contruct a family $f_\lambda$ 
of countably piecewise linear unimodal maps, for which $F_\lambda$ (see Figure~\ref{fig1}) are appropriate induced maps.
Both $f_\lambda$ and the induced map $F_\lambda$ are linear on intervals $W_k = [z_{k-1}, z_k]$ and $\hat W_k = [\hat z_k, \hat z_{k-1}]$
of length $\frac{1-\lambda}{2} \lambda^k$. 
Here $\hat x = 1-x$
is the symmetric image of a point or set, and $z_k < c < \hat z_k$ 
are the points in $f^{-S_k}(c)$ that are closest to $c$.
We define $F_\lambda(x) = f^{S_{k-1}}$ if $x \in W_k \cup \hat W_k$. 
The induced map $F_\lambda$ satisfies
\begin{equation}
\begin{CD}
[z_0, \hat z_0] @ >F_\lambda >>[z_0, \hat z_0]\\
@V\pi VV @VV\pi V\\
[0,1] @>T_\lambda >> [0,1]
\end{CD}
\qquad\qquad\qquad
\pi:x\mapsto \begin{cases}
\frac{1-2x}{2(1-z_0)} & \text{ if } x \le \frac12;\\[3mm]
\frac{2x-1}{2(1-z_0)} & \text{ if }  x \ge \frac12.
\end{cases}
\label{eq:comm_diag}
\end{equation}
Note that $\pi^{-1}(V_i)=W_i\cup\hat W_i$.

The one-parameter system $(Y,F_\lambda)$ is of
interest both for its own sake, see \cite{BT3,StraVogt}, and
for the sake of studying (thermodynamic properties of) $f$ itself.
Theorem~\ref{mainthm:Fibo} replaces the somewhat hypothetical picture of
smooth Fibonacci maps
with precise values of critical orders $\ell = \ell(\lambda)$,
where each of the different behaviours occurs.
In this non-differentiable setting, the critical order $\ell$ is defined
by the property that $\frac1C |x-c|^\ell < |f(x)-f(c)| \le C |x-c|^\ell$ for some $C > 0$ and
all $x \in [0,1]$.

\begin{maintheorem}\label{mainthm:Fibo}
The above countably piecewise linear unimodal map 
$f_\lambda$ (\ie with $|W_k| = |\hat W_k| = \frac{1-\lambda}{2} \lambda^k$
and  $\lambda \in (0,1)$) 
satisfies the following properties:
\begin{enumerate}[label=({\alph*}),  itemsep=0.0mm, topsep=0.0mm, leftmargin=7mm]
\item The critical order $\ell = 3 + \frac{2\log(1-\lambda)}{\log \lambda}$.
\item If $\lambda \in (\frac12,1)$, \ie $\ell > 5$,
then $f_\lambda$ has a wild attractor.
\item If $\lambda \in \left[\frac{2}{3+\sqrt5},\frac12\right]$, \ie $4 \le \ell \le 5$,
then $f_\lambda$ has no wild attractor, but an 
infinite $\sigma$-finite acim.
\item If $\lambda \in (0,\frac{2}{3+\sqrt5})$, \ie $\ell \in (3, 4)$,
then $f_\lambda$ has an acip.
\end{enumerate}
\end{maintheorem}

As above, let $\phi_t = -t \log|f_\lambda'|$ and $\Phi_t = -t \log|F'_\lambda|$ be the \emph{geometric}
potentials for the unimodal map $f_\lambda$ and its induced version $F_\lambda$,
respectively. (Note that $\Phi_t = \sum_{j=0}^{\tau-1} \phi_t \circ f_\lambda^j$
for inducing time $\tau = \tau(x)$, justifying the name \emph{induced potential}.)

In \cite{BT3}, the precise form of the pressure
function for $((0,1], T_\lambda, -t \log|T'_\lambda|)$
and therefore also for the system 
$(Y, F_\lambda, -t \log|F'_\lambda|)$, is given.  However, this is of
lesser concern to us here, because given $([0,1], f_\lambda)$ with potential $-t\log|f_\lambda'|$,  for most results on the induced system $(Y, F_\lambda)$ to transfer to back to the original system, the correct induced potential on $Y$ is $-\log|F'_\lambda| - p\tau$, where the shift $p\tau$ is determined by a constant $p$ (usually the pressure of $-t\log|f_\lambda'|$) and the 
inducing time $\tau$ where $\tau(x) = S_{k-1}$ whenever 
$x \in W_k \cup \hat W_k$.  The fact that the shift by $p\tau$ depends on the interval $k$ increases the complexity of this problem significantly.
Results from \cite{BT3} which apply directly are contained in the following theorem.
\begin{maintheorem}\label{thm:main hyp dim}
Let $\Bas_\lambda = \{ x \in I : f^n_\lambda(x) \to \omega(c) \text{ as }
n \to \infty\}$ be the basin of $\omega(c)$, and let the 
{\em hyperbolic dimension}
be the supremum of Hausdorff dimensions of hyperbolic sets $\Lambda$, \ie
$\Lambda$ is $f_\lambda$-invariant, compact but bounded away from $c$.
Then
$$
\dim_{hyp}(f_\lambda) = \dim_H(\Bas_{1-\lambda}) = t_1 
$$
where
\begin{equation}\label{eq:t1}
t_1 := \begin{cases}
 1 &\text{ if } \lambda\in (0,1/2],\\
 t_2 & \text{ if } \lambda\in[1/2, 1),
 \end{cases}
\qquad  \text{ where }\ t_2 := -\log 4/\log [\lambda(1-\lambda)].
\end{equation}
\end{maintheorem}

For the properties of pressure presented in
Theorem~\ref{mainthm:thermo_Fibo} and the related 
results in Section~\ref{sec:conformal_for_f}, it is advantageous to
use a different approach to pressure, called \emph{conformal pressure} $\Pconf(\phi_t)$,
which is the smallest potential shift allowing the existence of a 
conformal measure for the potential.
We refer Sections~\ref{sec:original} and \ref{sec:conformal_for_f}
for the precise definitions, but in Theorem~\ref{mainthm:thermo_Fibo_global} we will show that
conformal pressure coincides with the (variational) pressure defined in
\eqref{eq:pressure}.
In \cite{BT3}, it is shown that $t_1$ from \eqref{eq:t1}
is the smallest value at which the pressure 
$P(\Phi_t)$ of the induced system $(Y, F_\lambda,\Phi_t)$ becomes zero. 
This gives the background information for our third main theorem.

\begin{maintheorem}\label{mainthm:thermo_Fibo_global}
The countably piecewise linear Fibonacci map $f_\lambda$, $\lambda \in (0,1)$,
with potential $\phi_t$ has the following thermodynamical properties.
\begin{enumerate}[label=({\alph*}),  itemsep=0.0mm, topsep=0.0mm, leftmargin=7mm]
\item The conformal and variation pressure coincide: $\Pconf(\phi_t)=P(\phi_t)$;
\item For $t < t_1$, there exists a unique equilibrium state $\nu_t$ for $(I, f_\lambda, \phi_t)$; this is absolutely continuous w.r.t.\ the appropriate conformal measure $n_t$. For $t > t_1$, the unique equilibrium state for $(I, f_\lambda, \phi_t)$ is $\nu_\omega$,
the measure supported on the critical omega-limit set $\omega(c)$.
 For $t = t_1$, $\nu_\omega$ is an equilibrium state, and if 
$\lambda \in (0, \frac2{3+\sqrt 5})$ then so is the acip, denoted $\nu_{t_1}$;
\item The map $t\mapsto P(\phi_t)$ is real analytic on $ (-\infty, t_1)$. 
Furthermore $P(\phi_t) > 0$ for $t < t_1$ and $P(\phi_t) \equiv 0$
 for $t \ge t_1$, so there is a phase transition at $t = t_1$.
\end{enumerate}
\end{maintheorem}

Let $\gamma := \frac12(1+\sqrt{5})$ be the golden ratio
and $\Gamma := \frac{2 \log \gamma}{\sqrt{-\log[\lambda(1-\lambda)]}}$.
More precise information on the shape of the pressure function is
the subject of our fourth main result.

\begin{maintheorem}\label{mainthm:thermo_Fibo}
The pressure function $P(\phi_t)$ of the countably piecewise linear Fibonacci map $f_\lambda$, $\lambda \in (0,1)$,
with potential $\phi_t$ has the following shape:
\newcounter{Lcount2}
\begin{list}{\alph{Lcount2})}
{\usecounter{Lcount2} \itemsep 1.0mm \topsep 0.0mm \leftmargin=7mm}
\item 
On a left neighbourhood of $t_1$, 
there exist $\tau_0 = \tau_0(\lambda), \tau_0' = \tau_0'(\lambda) > 0$ such that
$$
P(\phi_t) > 
\begin{cases}
\tau_0  e^{-\pi \frac{\Gamma}{\sqrt{t_1-t}}} & \text{ if } t < t_1 \le 1 \text{ and } \lambda \ge \frac12;\\
\tau_0' (1-t)^{ \frac{\log \gamma}{\log R} } & \text{ if } t < 1 \text{ and } \frac{2}{3+\sqrt 5} \le \lambda < \frac12,
\end{cases} 
$$
where $R = \frac{ \left(1+\sqrt{1-4\lambda^t(1-\lambda)^t}\right)^2 }{4\lambda^t(1-\lambda)^t}$ and $\lim_{t \to 1} \log R \sim 2(1-2\lambda)$ for 
$\lambda \sim \frac12$.
\item On a left neighbourhood of $t_1$, 
there exist $\tau_1 = \tau_1(\lambda), \tau_1' = \tau_1'(\lambda)  > 0$ such that
$$
P(\phi_t) < 
\begin{cases}
\tau_1 e^{-\frac56 \frac{\Gamma}{\sqrt{t_1-t}} } & \text{ if } t < t_1 \le 1 \text{ and } \lambda \ge \frac12;\\
\tau_1' (1-t)^{ \frac{\lambda \log \gamma}{2t(1-2\lambda)} } & \text{ if } t < 1 \text{ and } \frac{2}{3+\sqrt 5} \le \lambda < \frac12.
\end{cases} 
$$
\item If $\lambda \in (0, \frac{2}{3+\sqrt 5})$, 
then $\lim_{s \uparrow t_1}\frac{d}{ds} P(\phi_s) < 0$;
otherwise (\ie if $\lambda \in [\frac{2}{3+\sqrt 5},1)$),  $\lim_{s\uparrow t_1} P(\phi_s) = 0$.
\end{list}
\end{maintheorem}

To put these results in context, let us discuss the results of Lopes
\cite[Theorem 3]{Lopes} on the thermodynamic behaviour of the 
Manneville-Pomeau map $g:x \mapsto x+x^{1+\alpha} \pmod 1$. 
The pressure function for this family is
$$
P(-t \log g') = \begin{cases}
\lambda(\mu_{ac})(1-t) + B(1-t)^{1/\alpha} + \text{ h.o.t.} & \text{ if } t < 1 \text{ and } \alpha \in (\frac12,1);\\
 C(1-t)^\alpha + \text{ h.o.t.} & \text{ if } t < 1 \text{ and } \alpha > 1;\\
0 & \text{ if } t \ge 1,
\end{cases}
$$
where $B, C > 0$ are constants, and $\lambda(\mu_{ac})$ is the Lyapunov exponent of the non-Dirac equilibrium state (\ie the acip). Hence the left derivative of the pressure at $t=1$ when $\alpha\in (1/2, 1)$ is $-\lambda(\mu_{ac})$.
Recall that in the acip case, due to Ledrappier's result \cite{Ledrap}, $h_{\mu_{ac}}=\lambda(\mu_{ac})$.
 Note that the transition case $\alpha=1$ corresponds to the transition
 from a finite acip (for $\alpha < 1$) to an infinite acim (for $\alpha \ge 1$).\footnote{The asymptotics of $P(t)$ in \cite[Theorem 3]{Lopes}
 don't hold for $\alpha = 1$ (personal communication  with A.O.\ Lopes), but
 since there is no acip, $P(t)$ is differentiable at $t=1$ with derivative $P'(t)=0$
 as in \cite{ITeqnat}. We don't know the higher order terms in this case.
Asymptotics of related systems are obtained
in \cite{PFK, BFKP}, namely for the Farey map $x \mapsto \frac{x}{1-x}$ if $x \in [0,\frac12]$ and $x \mapsto \frac{1-x}{x}$ if $x \in [\frac12,1]$. 
It is expected that their asymptotics also hold for the Manneville-Pomeau map 
with $\alpha = 1$.
In \cite{BLL}, a Manneville-Pomeau-like map with two neutral fixed points, both with
$\alpha = 1$, is considered, using a Hofbauer-like potential.}  
In the Manneville-Pomeau case there is no transition of Lebesgue measure changing from conservative to dissipative.
The phase transition at $t = 1$ is said to be of {\em first type} if there are
two equilibrium states (here an acip and the Dirac measure $\delta_0$);
if there is only one equilibrium state, then the phase transition 
is of {\em second type}. The exponent $1/\alpha$ is called the 
{\em critical exponent of transition}.

For Fibonacci maps, instead of a Dirac measure, there is a unique
measure $\nu_\omega$ supported on the critical $\omega$-limit set; 
it has zero entropy and Lyapunov exponent. 
Theorem~\ref{mainthm:thermo_Fibo} paints a similar picture to Lopes' result
for Manneville-Pomeau maps. In detail, we have

\begin{list}{$\bullet$}
{ \itemsep 0.5mm \topsep 0.0mm \leftmargin=7mm}
\item a phase transition of first type for $\lambda\in (0, \frac2{3+\sqrt 5})$: 
the pressure is not $C^1$ at
$t=t_1$. This is precisely the region from Theorem ~\ref{mainthm:Fibo} where
$f_\lambda$ has an acip $\mu_{ac}$, in accordance with the results from
\cite{ITeqnat}. According to Ledrappier \cite{Ledrap}, 
$h_\mu = \lambda(\mu_{ac})$ is the Lyapunov exponent, so
$\lim_{s \uparrow 1} \frac{d}{ds}P(\phi_s) = -\lambda(\mu_{ac})$. 
 Lebesgue measure is conservative here.

\item a phase transition of second type (with unique equilibrium state $\nu_\omega$ supported on $\omega(c)$) for $\lambda\in (\frac2{3+\sqrt 5}, \frac12)$: there is some minimal $n\in
\N$ such that the $n$-th left derivative $\lim_{t \uparrow t_1} \frac{d^n}{dt^n} P(\phi_t)<0$.
Thus the pressure function is $C^{n-1}$, but not $C^n$, at $t=t_1$ and
so there is an $n$-th order phase transition. 
Consequently, the critical exponent of
transition tends to infinity as $\lambda\nearrow 1/2$.  Lebesgue is still conservative here, and also for $\lambda=1/2$.

\item  a phase transition of second type
for $\lambda\in [1/2, 1)$: the pressure is $C^1$ with
$\frac{d}{dt}P(\phi_t)=0$ at $t=t_1$. By convexity, also
$\frac{d^2}{dt^2}P(\phi_t)=0$ at $t=t_1$. It is unlikely, but we cannot
a priori rule out, that the higher derivatives oscillate rapidly,
preventing the pressure function from being $C^\infty$ at $t=t_1$.  Lebesgue is dissipative for $\lambda\in (1/2, 1)$.
\end{list}

This paper is organised as follows.
In Section~\ref{sec:cpl} we introduce the countably piecewise linear 
unimodal maps and give conditions under which they produce an
induced Markov map that is linear on each of its branches.
In Section~\ref{sec:Fibo} this is applied to Fibonacci maps, and, using 
a random walk argument, the existence of an attractor and hence
Theorem~\ref{mainthm:Fibo} is proved.
Rather as an intermezzo, Section~\ref{sec:example} shows that for countably piecewise linear 
unimodal maps with infinite critical order, wild attractors
do exist beyond the Fibonacci-like combinatorics.
In Section~\ref{sec:original} 
we explain how conformal and invariant measures
of the induced system relate to conformal and invariant measures of 
the original system.
In Section~\ref{sec:conformal} we discuss the technicalities that the $2$-to-$1$ factor map from \eqref{eq:comm_diag}
poses for invariant and conformal measures; we also prove Theorem~\ref{thm:main hyp dim}.
The properties of the conformal pressure functions (existence, upper/lower bounds and nature of phase transitions)
are studied in Section~\ref{sec:conformal_for_f}; this section contains the 
proof the main part of Theorem~\ref{mainthm:thermo_Fibo}.
In Section~\ref{sec:invariant} we
prove the existence and properties of invariant measures
that are absolutely continuous w.r.t.\ the relevant conformal measures.
In the final section we present some general theory on 
countable Markov shifts due to Sarig.
This  leads up to the proof of
Theorem~\ref{mainthm:thermo_Fibo_global}, and also 
gives the final ingredient of the proof of Theorem~\ref{mainthm:thermo_Fibo}.

\section{The countably piecewise linear model}\label{sec:cpl}

Let $\N = \{ 1, 2, 3, 4, \dots \}$ and $\N_0 = \N \cup \{ 0 \}$.
Throughout $f:I \to I$ stands for a
symmetric unimodal map with unit interval $I = [0,1]$, critical point $c = \frac12$,
and $f(0) = f(1) = 0$. For $x \in [0,1]$, let $\hat x = 1-x$
be the point with the same $f$-image as $x$.
We use the same notation for sets.

Let us start by some combinatorial notation. For $n \ge 1$,
the {\em central branch} of $f^n$ is the restriction
of $f^n$ to any of the two largest one-sided neighbourhoods of $c$ on which
$f^n$ is monotone.
Due to the symmetry, the image of the left and 
right central branch is the same, and if it contains the critical point, 
then we say that $n$ is a {\em cutting time}.
We enumerate cutting times as $1 = S_0 < S_1 < S_2 < \dots$
If $f$ has no periodic attractors, $S_k$ is well-defined for all $k$,
and we will denote the point in the left (resp.\ right) central branch
of $f^{S_k}$ that maps to $c$ by $z_k$ (resp.\ $\hat z_k$).
These points are called the {\em closest precritical points}
and it is easy to see that the domains of the left (resp.\ right)
central branch of $f^{S_k}$ are $[z_{k-1}, c]$ (resp.\ $[c, \hat z_{k-1}]$).

The difference of two consecutive cutting times is again a cutting time.
Hence  (see \cite{Hof}) we can define the {\em kneading map}
$Q: \N \to \N_0$ by
\begin{equation*}%\label{eq:EasyHofb}
S_k - S_{k-1} = S_{Q(k)}.
\end{equation*}
A kneading map $Q$ corresponds to a sequence of cutting times
of a unimodal map if and only if it satisfies
\begin{equation}\label{eq:DifficultHofb}
\{ Q(k+j)\}_{j \ge 1} \succeq \{ Q(Q^2(k)+j)\}_{j \ge 1},
\end{equation}
for all $k \ge 1$, where $\succeq$ indicates lexicographical order (see \cite{Hof80}). Note that \eqref{eq:DifficultHofb} holds automatically if the 
kneading map is non-decreasing.

The construction of our unimodal map $f$ proceeds along the following steps:
\begin{enumerate}[label=({\Roman*}),  itemsep=0.0mm, topsep=0.0mm, leftmargin=7mm]
\item First fix a kneading map $Q$ such that
\begin{equation}
Q(k+1) > Q(Q^2(k)+1) \label{12.1}
\end{equation}
for every $k \ge 2$. This is obviously stronger than \eqref{eq:DifficultHofb}, but provides a considerable simplification
of the proof.
%The general case should also be possible, but the construction is more complicated.
\item
By convention, set $z_{-1}=0$ and $\hat z_{-1}=1$.
For $j \ge 0$, choose a strictly increasing sequences of points $z_j \nearrow c = \frac12$ and 
$\hat z_j = 1-z_j \searrow c$. (The points $z_j$ will play the role of the closest precritical
points, cf.\ ($\mbox{IH}_j$) in the proof of Proposition~\ref{Proposition 12.1}.) 
Set 
\begin{equation*} \label{12.2}
 W_j:=(z_{j-1} ,z_j), \ \hat W_j:=(\hat z_j,\hat z_{j-1}) \text{ and  }
\eps_j := |W_j| = |\hat W_j| > 0.
\end{equation*}
Therefore, $\sum_{j \ge 0} \eps_j = \frac12$.
\item Define
\begin{equation}
s_j := \frac1{\eps_j} \sum_{i \ge Q(j)+1} \eps_i
= \frac{|z_{Q(j)}-c|}{|z_j-z_{j-1}|}, \label{12.4}
\end{equation}
for $j \ge 1$; these numbers will turn out to be the absolute values of the 
slopes of $F|_{W_j}$ for the induced map $F$, see \eqref{12.3} below.
\item
For $j \ge 0$, we define numbers $\kappa_j > 0$ 
that will represent the slope of $f|_{W_j}$. Let
\begin{equation}\label{12.5}
\kappa_0 := \frac1{2\eps_0}.
\end{equation}
(This will give that $f(z_0)= \kappa_0 \cdot (z_0-z_{-1}) = \frac12 = c$.)
Next, set
\begin{equation} \label{12.6}
\kappa_1 := s_1 = \frac1{\eps_1} \sum_{i \ge 1} \eps_i =
\frac{1-2\eps_0}{2 \eps_1}.
\end{equation}
(Since inducing time $S_0 = 1$ on $W_1$, it makes sense that
the slopes of $f$ and $F$ on $W_1$ are the same. In fact, we will
have $F|_{W_1}=f^{S_0}|_{W_1}=f|_{W_1}$.)
For $j \ge 2$, we set inductively
\begin{equation}
\kappa_j := \begin{cases}
\frac{s_j}{\kappa_0} \frac{\kappa_{j-1}}{s_{j-1}}
&\text{ if } Q(j-1) = 0, \\
\frac{s_j \cdot \kappa_{j-1}}
{s_{j-1} \cdot s_{Q(j-1)} \cdot s_{Q^2(j-1)+1}} &\text{ if } Q(j-1) > 0.
\end{cases} 
\label{12.7}
\end{equation}
\item
Let $f$ be the unique continuous unimodal map such that
$$
\left\{ \begin{array}{l}
f(z_{-1}) = f(\hat z_{-1}) = z_{-1} \\
Df|_{W_j} = - Df|_{\hat W_j} = \kappa_j,
\end{array} \right.
$$
so that $|f(W_i)| = \kappa_i \eps_i$ and each interval $f(W_i)$ is
adjacent to $f(W_{i+1})$.
\end{enumerate}
Thus $f$ is completely determined by the choice of $Q$ and points $z_j$.
In Section~\ref{sec:Fibo}
on Fibonacci combinatorics, we let $z_j \nearrow c$ in a geometric manner, 
or precisely,
$\eps_j = \frac{1-\lambda}{2} \lambda^j$ so that $f$ depends solely on a 
the single parameter $\lambda \in (0,1)$.
In this section, we will continue with the more general set-up.

The induced map\footnote{In later sections, the interval on which the induced map is defined will be called $Y$.} is defined as: 
\begin{equation}\label{12.3}
F:(z_0,\hat z_0) \to (z_0,\hat z_0),
\qquad F|_{W_j \cup \hat W_j} = f^{S_{j-1}}|_{W_j \cup \hat W_j}
\text{ for } j \ge 1.
\end{equation}
Since the $z_j$ will play the role of the closest precritical
points, we will have $f^{S_{j-1}}(z_j) = f^{S_{j-1}}(\hat z_j)
\in \{ z_{Q(j)}, \hat z_{Q(j)} \}$,
and therefore
\[
F(W_j) = F(\hat W_j) = \cup_{i > Q(j)} W_i \text{ or }
\cup_{i > Q(j)} \hat W_i.
\]
In Proposition~\ref{Proposition 12.1}, we will prove that $F|_{W_j}$ and $F|_{\hat W_j}$ 
are also linear.

We pose two other conditions on the sequence $(\eps_j)_{j \in \N}$,
which will be checked later on for specific examples, in particular 
the Fibonacci map. Let $x^f = f(x)$ for any point $x$.
For all $j \ge 2$:
\begin{equation}
\frac{s_j}{\kappa_j} |c^f - z^f_j|  =
\frac{s_j}{\kappa_j} \sum_{i = j+1}^{\infty} \kappa_i \eps_i
\le \eps_{Q(j)}, \label{12.8}
\end{equation}
and
\begin{equation}
\frac{s_j}{\kappa_j} |c^f - z^f_j|  =
\frac{s_j}{\kappa_j} \sum_{i = j+1}^{\infty} \kappa_i \eps_i
\le \frac{\eps_{Q^2(j)+1}}{s_{Q(j)}} \quad \text{ whenever } Q(j) > 0.
\label{12.9}
\end{equation}

\begin{proposition}\label{Proposition 12.1}
Let $f$ be the map constructed above, \ie assume that
\eqref{12.1}-\eqref{12.9} hold.
Then $Q$ is the kneading map of $f$,
and the induced map $F$ is linear on each set $W_j$ and $\hat W_j$, having
slope $\pm s_j$.
\end{proposition}

\begin{proof}
We argue by induction, using the induction hypothesis, for $j \ge 2$,
\begin{equation*}
\left\{ \begin{array}{l}
f^{S_{j-1}-1}|_{(c^f,z^f_{j-1})} \text{ is linear, with slope }
\frac{s_j}{\kappa_j}. \\
f^{S_{j-1}}(z_{j-1}) = c. \\
f^{S_{j-1}}(c) \in W_{Q(j)} \text{ or } \hat W_{Q(j)}.
\end{array} \right. \qquad (\mbox{IH}_j)
\end{equation*}
From the first statement, it follows immediately that
\begin{equation}
f^{S_{j-1}}|_{W_j} \text{ is linear, with slope } s_j,
\text{ for } j \ge 1. \label{12.10}
\end{equation}
%which we would naturally expect from the construction of the
%induced map. ($f|_{W_0}$ is linear, with slope $\kappa_0$.)
From this and the fact that $f^{S_{j-1}}(z_{j-1}) = c$,
it follows that
\begin{equation}
f^{S_{j-1}}(z_j) = f^{S_{j-1}}(z_{j-1}) \pm s_j \eps_j =
c \pm \sum_{i \ge Q(j)+1} \eps_i = z_{Q(j)} \text{ or }
\hat  z_{Q(j)}. \label{12.11}
\end{equation}

Let us prove $(\mbox{IH}_j)$ for $j=2$.
It is easily checked that $f(z_0)=f(\hat z_0)=c=\frac12$, and hence $f(c) \in
\hat W_0$.
$f(z_1) = \frac12+\kappa_1\eps_1
= \frac12 + \frac12 - \eps_0 = \hat z_0$.
So $f^{S_1}(z_1) = c$ and because $c^f \in \hat W_0$,
$f^{S_1-1}|_{(c^f,z^f_1)} = f|_{(c^f,\hat z_0)}$
is also linear, with slope $\kappa_0=\frac{s_2}{\kappa_2}$.
Next we check the position of $f^{S_1}(c)$. By the above formula, and the
additional assumption \eqref{12.8},
\begin{eqnarray*}
f^{S_1}(c) &=& f^{S_1}(z_2) - |f^{S_1-1}((c^f,z^f_2))| \\
&=& z_{Q(2)}-\frac{s_2}{\kappa_2}|c^f-z^f_2| \ge z_{Q(2)} - \eps_{Q(2)}
= z_{Q(2)-1}.
\end{eqnarray*}
Hence $f^{S_1}(c) \in W_{Q(2)}$.

Next assume that $(\mbox{IH}_i)$ holds for
$i < j$.
Using \eqref{12.11} and $(\mbox{IH}_{Q(j-2)})$ subsequently, we get
$$
f^{ S_{j-1} }(z_{j-1}) = f^{ S_{Q(j-1)} } \circ f^{S_{j-2}}(z_{j-1}) =
f^{S_{Q(j-1)}}(z_{Q(j-1)}) = c.
$$
Because $(c^f,z^f_{j-1}) \subset (c^f,z^f_{j-2})$, $(\mbox{IH}_{j-1})$
yields that
$$
f^{S_{j-2}-1}|_{(c^f,z^f_{j-1})} \text{ is linear with slope }
\frac{s_{j-1}}{\kappa_{j-1}}.
$$
By \eqref{12.11} and $(\mbox{IH}_{j-1})$, its image is the interval
$(z_{Q(j-1)},c_{S_{j-2}}) \subset W_{Q(j-1)}$ or $\hat W_{Q(j-1)}$.
Now if $Q(j-1)=0$, then
$$
f^{S_{j-1}-1}|_{(c^f,z^f_{j-1})} = f \circ f^{S_{j-2}-1}|_{(c^f,z^f_{j-1})}
\text{ is linear with slope }
\kappa_0 \frac{s_{j-1}}{\kappa_{j-1}}. $$
By the first part of the definition of $\kappa_j$, this slope is equal to
$\frac{s_j}{\kappa_j}$.
If $Q(j-1) > 0$ then
$$
f^{S_{j-1}-1}|(c^f,z^f_{j-1}) =
f^{S_{Q^2(j-1)}} \circ f^{S_{Q(j-1)-1}} \circ f^{S_{j-2}-1}|_{(c^f,z^f_{j-1})}.
$$

\begin{figure}[ht]
\unitlength=10mm
\begin{picture}(11,8)(0,0) \let\ts\textstyle
%first bar
\put(0,7){\line(1,0){11}}
\put(4,7){\line(0,1){0.1}} \put(3.9,6.7){$c$}
\put(7,7){\line(0,1){0.1}} \put(6.8,6.7){$\hat z_j$}
\put(10,7){\line(0,1){0.1}} \put(9.8,6.7){$\hat z_{j-1}$}
%second bar
\put(0,5){\line(1,0){11}}
\put(1,5){\line(0,1){0.1}} \put(0.7,4.7){$z_{Q(j-1)-1}$}
\put(4,5){\line(0,1){0.1}} \put(3.7,4.7){$c_{S_{j-2}}$}
\put(7,5){\line(0,1){0.1}}
\put(10,5){\line(0,1){0.1}} \put(9.8,4.7){$\hat z_{Q(j-1)}$}
%third bar
\put(0,3){\line(1,0){11}}
\put(1,3){\line(0,1){0.1}} \put(0.9,2.7){$c$}
\put(2.5,3){\line(0,1){0.1}} \put(1.7,2.7){$\hat z_{Q^2(j-1)+1}$}
\put(4,3){\line(0,1){0.1}} \put(3.5,2.7){$c_{S_{j-1}-S_{Q^2(j-1)}}$}
\put(7,3){\line(0,1){0.1}}
\put(10,3){\line(0,1){0.1}} \put(9.7,2.7){$\hat z_{Q^2(j-1)}$}
%fourth bar
\put(0,1){\line(1,0){11}}
\put(1,1){\line(0,1){0.1}} \put(0.7,0.7){$c_{S_{Q^2(j-1)}}$}
\put(3,1){\line(0,1){0.1}} \put(2.5,0.7){$z_{Q(j)-1}$}
\put(4,1){\line(0,1){0.1}} \put(3.8,0.7){$c_{S_{j-1}}$}
\put(7,1){\line(0,1){0.1}} \put(6.8,0.7){$z_{Q(j)}$}
\put(10,1){\line(0,1){0.1}} \put(9.9,0.7){$c$}
%arrows
\put(0,6.7){\vector(0,-1){1.4}} \put(0.2,6){$f^{S_{j-2}}$}
\put(0,4.7){\vector(0,-1){1.4}} \put(0.2,4){$f^{S_{Q(j-1)-1}}$}
\put(0,2.7){\vector(0,-1){1.4}} \put(0.2,2){$f^{S_{Q^2(j-1)}}$}
\end{picture}
\caption{Position of various precritical points and their images.}
\label{fig2}
\end{figure}
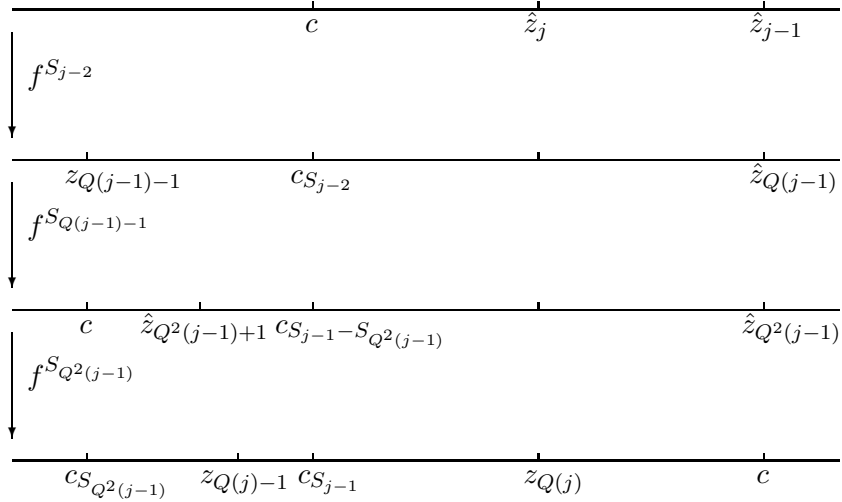

By \eqref{12.10}, $f^{S_{Q(j-1)-1}}|_{W_{Q(j-1)}}$ is linear with slope
$s_{Q(j-1)}$.
Hence
$f^{S_{Q(j-1)-1}} \circ f^{S_{j-2}-1}|_{(c^f,z^f_{j-1})}$
is linear with slope
$s_{Q(j-1)} \frac{s_{j-1}}{\kappa_{j-1}}$.
By \eqref{12.11}, its image is the interval
$$
(z_{Q^2(j-1)},c_{S_{j-2}+S_{Q(j-1)-1}}) = (z_{Q^2(j-1)},c_{S_{j-1}-S_{Q^2(j-1)}}).
$$
By \eqref{12.9}, the length of this interval is
$|c^f-z^f_{j-1}| s_{Q(j-1)} \frac{s_{j-1}}{\kappa_{j-1}} \le
\eps_{Q^2(j-1)+1}$, so
$$
(z_{ Q^2(j-1) },c_{ S_{j-1}-S_{Q^2(j-1)} }) \subset W_{Q^2(j-1)+1}
\text{ or } \hat W_{ Q^2(j-1)+1 }.
$$
By \eqref{12.10}, $f^{S_{Q^2(j-1)}}|_{W_{Q^2(j-1)+1}}$ is also linear, with
slope $s_{Q^2(j-1)+1}$.
It follows that
$f^{S_{Q^2(j-1)}} \circ f^{S_{Q(j-1)-1}} \circ
f^{S_{j-2}-1}|_{(c^f,z^f_{j-1})}$ is linear with slope
$s_{Q^2(j-1)+1} s_{Q(j-1)} \frac{s_{j-1}}{\kappa_{j-1}}$.
The second part of \eqref{12.7} gives that
$f^{S_{j-1}-1}|(c^f,z^f_j)$ is linear with slope
$\frac{s_j}{\kappa_j}$,
as asserted.
By \eqref{12.8}, the length of the image is
$|c^f-z^f_j| \frac{s_j}{\kappa_j} \le
\eps_{Q(j)}$.
Formula \eqref{12.11} yields $f^{S_{j-1}}(z_j) = z_{Q(j)}$. Hence we obtain
$$
z_{Q(j)} > f^{S_{j-1}}(c) \ge z_{Q(j)}-\eps_{Q(j)}
$$
or
$$\hat z_{Q(j)} < f^{S_{j-1}}(c) \le \hat z_{Q(j)}+\eps_{Q(j)}.$$
In other words, $f^{S_{j-1}}(c) \in W_{Q(j)}$ or $\hat W_{Q(j)}$.
This concludes the induction.
(Notice that
$\frac{|c_{S_{j-1}}-z_{Q(j)}|}{|z_{Q(j)-1}-z_{Q(j)}|} =
\frac1{\eps_{Q(j)}} \frac{s_j}{\kappa_j} |c^f-z^f_{j-1}|$.)
\end{proof}

\section{The Fibonacci case}\label{sec:Fibo}

In this section we prove Theorem~\ref{mainthm:Fibo}.
Let $\ph_n(x) = j$ if $F^n(x) \in W_j \cup \hat W_j$.
With respect to the existence of wild attractors and the random walk
generated by
$F$, we are in particular interested in the conditional expectation
(also called {\em drift}) 
\begin{equation}
\E(\ph_n-k \mid \ph_{n-1} = k) =
\frac{\sum_{i \ge Q(k)+1} (i-k)\eps_i}{\sum_{i \ge Q(k)+1} \eps_i}
= \frac{\sum_{i \ge Q(k)+1} i\eps_i}{\sum_{i \ge Q(k)+1} \eps_i}
- k.
\label{12.12}
\end{equation}
Drift in the setting of Fibonacci maps seems to be used first in \cite{KN}.
Note that here that the expectation is with respect to Lebesgue measure.

\begin{proof}[Proof of Theorem~\ref{mainthm:Fibo}]
We attempt to solve the problem for $\eps_j = |W_j| = |\hat W_j| = \frac{1-\lambda}{2} \lambda^j$, so $\sum_{j \ge 0} \eps_j = \frac12$.
By formula \eqref{12.4},
\begin{equation*}%\label{eq:slopes_s}
\begin{cases}
s_1 = \pm \frac1{1-\lambda} \\
s_j = \pm \frac1{\lambda(1-\lambda)} \text{ for } j \ge 2.
\end{cases}
\end{equation*}
(Note that the slopes $s_j \ge 4$, with the minimum assumed at $\lambda = \frac12$.)
Using \eqref{12.7}, we obtain for the slope $\kappa_j = f'(x)$, $x \in W_j$.
\begin{equation}\label{eq:kappa}
\kappa_j = \begin{cases} 
\frac1{1-\lambda} & j = 0, 1; \\
\frac1{\lambda} & j = 2; \\
\frac{(1-\lambda)}{\lambda} &  j = 3; \\
\frac{(1-\lambda)^3}{\lambda} &  j = 4; \\
\frac{\lambda^{2j}(1-\lambda)^{2j}}{\lambda^{10}(1-\lambda)^5}
& j \ge 5.
\end{cases}
\end{equation}
Let us first check \eqref{12.8} and \eqref{12.9}.
For simplicity, write $\eps_j = C_1 \lambda^j$ and
$\kappa_j = C_2 \omega^j$ where $\omega=\lambda^2(1-\lambda)^2$. Then 
\begin{eqnarray*}
\frac{s_j}{\kappa_j} \sum_{i = j+1}^{\infty} \kappa_i \eps_i
\le \eps_{Q(j)}  
&\equi& \sum_{i = j+1}^{\infty} C_1C_2 (\lambda \omega)^i
\frac1{\lambda(1-\lambda)} \frac1{C_2 \omega^j}
\le C_1 \lambda^{j-2} \\ 
&\equi& \frac{\lambda^{j+1}\omega^{j+1}}{1-\lambda\omega}
\frac1{\lambda(1-\lambda) \omega^j}
\le \lambda^{j-2} \\[2mm] 
&\equi&\lambda^4(1-\lambda) \le 1-\lambda^3(1-\lambda)^2.
\end{eqnarray*}
This is true for every $\lambda \in (0,1)$.
Checking \eqref{12.9} for $Q(j) > 0$, we get
\begin{eqnarray*}
\frac{s_j}{\kappa_j} \sum_{i = j+1}^{\infty} \kappa_i \eps_i
\le \frac{\eps_{Q^2(j)+1}}{s_{Q(j)}}  &\equi&
\frac{\lambda^{j+1}\omega^{j+1}}{1-\lambda\omega}
\frac1{\lambda(1-\lambda) \omega^j}
\le \lambda^{j-3}\lambda(1-\lambda) \\
&\equi&\lambda^4 \le 1-\lambda^3(1-\lambda)^2. 
\end{eqnarray*}
Again, this is true for all $\lambda \in (0,1)$.

Let us compute the order $\ell$ of the critical point.
Indeed, $|Df(x)| = O(\lambda^{2j}(1-\lambda)^{2j} )$
and $|x-c| = O(\lambda^j)$ if $x \in W_j$.
On the other hand $|Df(x)| = O(|x-c|^{\ell-1})$.
Therefore
$$
\ell = 1 + \frac{\log \omega}{\log \lambda} =
3 + \frac{2\log(1-\lambda)}{\log \lambda}.
$$
Consider \eqref{12.12} again. For $k \ge 2$, the {\em drift} is 
\begin{equation*}%\label{eq:drift}
\drift(\lambda) := \E(\ph_n-k \mid \ph_{n-1} = k) =
\frac{\sum_{i \ge k-1} i\eps_i}{\sum_{i \ge k-1} \eps_i} - k
= \frac{\lambda}{(1-\lambda)} - 1 = \frac{2\lambda-1}{1-\lambda}.
\end{equation*}
Hence $\E(\ph_n-k \mid \ph_{n-1} = k) > 0$ if
$\lambda > 1-\lambda$, \ie $\lambda > \frac12$.
The second moment
$$
\frac{\sum_{i \ge Q(k)+1} (i-k)^2\eps_i}{\sum_{i \ge Q(k)+1}
\eps_i} =
\frac{\lambda^2}{(1-\lambda)^2} - 2\frac{\lambda}{1-\lambda} + 1
$$
is uniformly bounded, and therefore also the variance.
So as in the proof of \cite[Theorem 1]{BT3}, for $\lambda > \frac12$,
\ie a critical order larger than $5$, the
Fibonacci map $f$ exhibits a wild attractor. 

Now we will calculate for what values of $\lambda$, $f$ has an infinite
$\sigma$-finite measure.
First take $\lambda < \frac12$.
Then $F$ (considered as a Markov process) is recurrent, and therefore has
an invariant probability measure $\mu$.
Let $(A_{i,j})_{i,j}$ be the transition matrix corresponding to $F$,
and let $(v_i)_i$ be the invariant probability vector, \ie
left eigenvector with eigenvalue $1$. As $F$ is a Markov map, and $F$ is
linear on each state $W_k$, we obtain $\mu(W_k) = v_k$.
So let us calculate this.

\begin{equation*}
A_{i,j} = \left\{ \begin{array}{ll}
0                              & \text{ if } j \le Q(i), \\
(1-\lambda)\lambda^{j-(Q(i)+1)}  & \text{ if } j > Q(i),
\end{array}\right.
\end{equation*}

or in matrix form
\begin{equation}
(A_{i,j})_{i,j} = (1-\lambda)
\left( \begin{array}{ccccccc}
1      & \lambda & \lambda^2 & \lambda^3 & \lambda^4 & \hdots   & \hdots \\
1      & \lambda & \lambda^2 & \lambda^3 & \lambda^4 & \hdots   & \hdots \\
0      & 1      & \lambda   & \lambda^2 & \lambda^3 & \lambda^4 & \hdots \\
0      & 0      & 1        & \lambda   & \lambda^2 & \lambda^3 & \hdots \\
\vdots & \vdots & 0        & 1        & \lambda   & \lambda^2 & \hdots \\
\vdots & \vdots & \vdots   & \vdots   & \vdots   & \vdots   & \ddots
\end{array} \right) .\label{eq:A}
\end{equation}

As in \cite[Theorem 1]{BT3}, this matrix has a unique normalised eigenvector:
\begin{equation}\label{eq:v}
v_i = \frac{1-2\lambda}{\lambda} \left( \frac{\lambda}{1-\lambda} \right)^i
\text{ for } \lambda < \frac12.
\end{equation}
According to \cite[Theorem 2.6]{Btams},
$f$ has a finite measure if and only if
\begin{equation} \label{12.14}
\sum_k S_{k-1} \mu(W_k) < \infty.
\end{equation}
If \eqref{12.14} fails, then $f$ has an absolutely continuous
$\sigma$-finite measure. This follows because $f$ is conservative, and
$\omega(c)$ is a Cantor set \cite{HK}.
In the Fibonacci case $S_{k-1} \sim \gamma^{k-1}$, where $\gamma =
\frac{1+\sqrt5}{2}$ is the golden mean. Since $\mu(W_k) = \beta_i \rho^i$
for $\beta_i \equiv \lambda$, as
we saw above, we obtain $\rho > \frac1{\gamma}$ if and only if
$\frac{1+\sqrt5}{2} \frac{\lambda}{1-\lambda} > 1$, \ie
$\lambda > \frac{2}{3+\sqrt5}$. This corresponds to the critical order
$\ell = 4$.
Therefore there exists a $\sigma$-finite measure for all
$\frac{2}{3+\sqrt5} \le \lambda < \frac12$, and a finite measure for
$0 < \lambda < \frac{2}{3+\sqrt5}$.
\end{proof}

\begin{remark} Since $c_{S_k} \in W_{k-1} \cup \hat W_{k-1}$ for
every $k \ge 1$,
we obtain $|Df^{S_{Q(k+1)}}( c_{S_k} )| =
|Df^{S_{k-1}}( c_{S_k} )| = \left|Df^{S_{k-1}}|_{W_k} \right| \cdot \left|Df^{S_{k-3}}|_{W_{k-2}} \right|
= [\lambda(1-\lambda)]^{-2}$.
Therefore
$|Df^{S_j}(c_1| \approx \kappa_j [\lambda(1-\lambda)]^{2j}
= \lambda^{-10} (1-\lambda)^{-5}$,
which is uniformly bounded in $j$.
Therefore the
Nowicki-van Strien summability condition (see \cite{NowStr}) fails for all $\lambda \in (0,1)$.
\end{remark}

\begin{remark}
As proved in \cite[Theorem B]{BT3}, $F_\lambda$ (or equivalently $T_\lambda$) is null recurrent w.r.t.\ Lebesgue 
when $\lambda = \frac12$.
\end{remark}

\section{An example of a wild attractor for $k-Q(k)$ unbounded}\label{sec:example}

In \cite{Btams} it was shown that smooth unimodal maps for which
$k-Q(k)$ is unbounded cannot have any wild attractors,
for any large but finite value of the critical order.
There are very few results known for unimodal maps
with flat critical points (\ie $\ell = \infty$), although we mention \cite{BenMis, Zwei_flat} and \cite{LevSw}, which deal with Lebesgue conservative Misiurewicz 
maps and infinitely renormalisable dynamics respectively.
The next example serves as a model for a unimodal map with infinite critical
order, suggesting that \cite[Theorem 8.1]{Btams} doesn't hold
anymore: There exists countably piecewise linear maps with kneading map
$Q(k) = \lfloor rk \rfloor$, $r \in (0,1)$ that have a wild attractor.
\\[2mm]
{\bf Example 1:}
Consider maps with kneading map
$$Q(k) =  \lfloor rk \rfloor$$
for some $r \in (0,1)$ and $k$ large. Here $ \lfloor x \rfloor$ indicates
the integer part of $x$. Since $Q$ is non-decreasing, \eqref{eq:DifficultHofb} holds and unimodal maps with this kneading map indeed exist.

Let $\alpha$ be such that $\frac{1}{\alpha-1} + \log r > 0$.
Take $\eps_k = C k^{-\alpha}$,
where $C$ is the appropriate normalising constant: $C \approx \alpha-1$.
This suffices to compute the expectation from \eqref{12.14}, at least for
large values of $k$. But instead of $\ph_n$, we prefer to look at
$\log \ph_n$. It is clear that $\log \ph_n(x) \to \infty$
if and only if $\ph_n(x) \to \infty$. So it will have the same
consequences. The advantage is that in this way we can keep the second
moment bounded.

We will calculate the expectation for large values of $k$. Therefore we
will write $rk$ for $Q(k)+i = \lfloor rk \rfloor + i$ and
$r^2k$ for $Q^2(k)+i =  \lfloor rk \rfloor + i$, where $i \in \{-1,0,1,2\}$.
We will also pass to integrals to simplify the calculations.
\begin{align*}
\E(\log \ph_n-\log k \mid \ph_{n-1} = k)
& =
\frac{\sum_{i \ge Q(k)+1} \eps_i \log i}
{\sum_{i \ge Q(k)+1} \eps_i}
- \log k \\
& \approx
\frac{\int_{rk}^{\infty} t^{-\alpha} \log t  dt}
{ \int_{rk}^{\infty} t^{-\alpha} dt}
- \log k \\
& = \frac1{\alpha-1} + \log r.
\end{align*}
This is positive by the choice of $\alpha$.
For the second moment we get
$$\alignedat1
&\E((\log \ph_n-\log k)^2 \mid \ph_{n-1} = k)
=
\frac{\sum_{i \ge Q(k)+1} (\log i - \log k)^2 \eps_i}
{\sum_{i \ge Q(k)+2} \eps_i} \\
&\approx
\frac{\int_{rk}^{\infty} t^{-\alpha} (\log t - \log k)^2  dt}
{ \int_{rk}^{\infty} t^{-\alpha} dt}
=
\log^2 r + \frac2{\alpha-1} \log r + \frac2{(\alpha-1)^2}.
\endalignedat$$
which is uniformly bounded in $k$.
Therefore, the induced map has drift to $c$, and thus is Lebesgue dissipative.

For the slopes of the induced map, and the original map we get the
following:
$$
s_j = \frac1{\eps_j} \sum_{i \ge Q(j)+1} \eps_i
\approx j^{\alpha} \int_{rj}^{\infty} t^{-\alpha} dt
= j \frac{r^{1-\alpha}}{\alpha-1},
$$
whence
$$
\kappa_j \approx \frac{\kappa_{j-1}}{s_{j-1}}
\frac{s_j}{s_{[rj \rfloor} s_{[r^2j \rfloor} } \approx
\kappa_{j-1} \frac{(\alpha-1)^2}{r^{5-2\alpha}} \frac1{j^2}
= O(B^j (j!)^{-2})
$$
for $B = \frac{(\alpha-1)^2}{r^{5-2\alpha}}$.
Next we check conditions \eqref{12.8} and \eqref{12.9}. Because
$\frac{\eps_{Q^2(j)+1}}{s_{Q(j)}} \le \eps_{Q(j)}$,
it suffices to check \eqref{12.9}. For $j$ sufficiently large,
\begin{eqnarray*}
\frac{s_j}{\kappa_j} \sum_{i = j+1}^{\infty} \kappa_i \eps_i
&\approx& j \frac{r^{1-\alpha}}{\alpha-1} \frac{(j!)^2}{B^j}
\left( \frac{B^{j+1}}{(j+1)!^2}C(j+1)^{-\alpha} +
\frac{B^{j+2}}{(j+2)!^2}C(j+2)^{-\alpha} + \cdots \right) \\
&\le& \frac{r^{1-\alpha}}{\alpha-1}
C B (j+1)^{-\alpha-1} \cdot r^{2-2\alpha} \\
&<& C (\alpha-1) r^{-\alpha-2} j^{-\alpha-1}
\approx \frac{\eps_{Q^2(j)+1}}{s_{Q(j)}}
\end{eqnarray*}
Hence, asymptotically there are no restrictions to
build a piecewise linear map for this kneading map.

The critical order of this map is infinite. Indeed, the slope on $(z_{j-1},z_j)$
is $\kappa_j \approx \frac{B^j}{(j!)^2}$.
$|c-z_{j-1}| = \sum_{i=j}^{\infty} \eps_j \approx (\alpha-1)
\int_j^{\infty} t^{-\alpha} dt = j^{1-\alpha}$.
So the critical order $\ell$ must satisfy
$$
\ell j^{(1-\alpha)(\ell-1)} = O\left(\frac{B^j}{(j!)^2}\right).
$$
This is impossible for finite $\ell$.

\section{Projecting thermodynamic formalism to the original system}
\label{sec:original}

In order to understand the thermodynamic properties of our systems $(I,f_\lambda)$ and $(Y, F_\lambda)$ more deeply, we need the definition of conformal measure.  Since we want to use this notion for both of these systems, we define it for general dynamical systems and potentials which preserve the Borel structure (so we implicitly assume our phase space is a topological space). 

\begin{definition}
Suppose that $g:X\to X$ is a dynamical system and $\phi:X\to [-\infty, \infty]$ is a potential, both preserving the Borel structure.  Then a measure $m$ on $X$ is called \emph{$\phi$-conformal} if for any measurable set $A\subset X$ on which $g:A\to g(A)$ is a bijection, 
$$m(g(A))=\int_Ae^{-\phi}~dm.$$
\label{def:conformal}
\end{definition}

For the geometric potential $\phi_t = -t\log|Df_\lambda|$ of the original system $(I,f_\lambda)$,
we want to determine for which
potential shift there is a $(\phi_t-p)$-conformal measure,
and potentially an invariant measure equivalent to it.
For a general potential $\phi$ for $(I,f_\lambda)$, the {\em induced potential}
is defined as
$$
\Phi(x) = \sum_{j=0}^{\tau(x)-1} \phi \circ f_\lambda(x),
$$
and hence it contains the inducing time in a fundamental way.
Even if $\phi$ is constant (or shifted by a constant amount $p$),
the induced potential is no longer constant (and shifted by $\tau p$).  
More concretely, for potential $\phi_t-p$, the induced potential is $ -t\log|F_\lambda'|-\tau p$, where $\tau p$ is 
the shift by the scaled inducing time $\tau_i = S_{i-1}$ on 
$W_i \cup \hat W_i$.  
In Lemma~\ref{lem:conformal-induced} below we prove the connection between a $(\phi_t -p)$-conformal measure for $(I, f_\lambda)$ and a $(\Phi_t-p\tau)$-conformal measure for $(Y, F_\lambda)$.
 
For $n\ge 1$ we define the set of \emph{$n$-cylinders} for $F_\lambda$ to be the collection of maximal intervals on which $F_\lambda^n$ is a homeomorphism. It is natural to denote such an $n$-cylinder 
by $C_{i_0\dots i_{n-1}}$, 
if for each $0\le k\le n-1$, $F_\lambda^k(C_{i_0\dots i_{n-1}}) \subset  W_{i_k}$ or $F_\lambda^k(C_{i_0\dots i_{n-1}}) \subset \hat W_{i_k}$.  
The sequence $i_0\cdots i_{n-1}$ is called the \emph{address} of the $n$-cylinder.  Observe that for each such address there are two $n$-cylinders:  we denote the one to the left of $c$ by $C_{i_0\cdots i_{n-1}}$ and that on the right by $\hat C_{i_0\cdots i_{n-1}}$, and let $(C\cup \hat C)_{i_0\ldots i_{n-1}}$ be the union of these.  Only certain sequences $i_0\cdots i_{n-1}$ can be realised as addresses, specifically we require $i_k\le i_{k-1}+1$ for $1\le k\le n-1$; we call such addresses \emph{admissible}.   
Notice that for any $x\in C_{i_0\dots i_{n-1}}$, $\tau^n(x)=S_{i_0}+\cdots + S_{i_{n-1}}$.  
Clearly cylinder sets can be defined analogously (without the ambiguity in address) for the map $T_\lambda$. 

As usual, the original system $(I,f)$ can be connected to the induced
system $(Y,F)$ via an intermediate tower construction,
say $(\Delta, f_{\Delta})$, defined as follows:
The space is the disjoint union
$$
\Delta = \bigsqcup_i \bigsqcup_{l = 0}^{\tau_i-1} \Delta_{i,l},
$$
where $\Delta_{i,l}$ are copies of $W_i$ and $\hat W_i$,
and the inducing time $\tau_i = \tau|_{W_i \cup \hat W_i} = S_{i-1}$. 
Points in $\Delta_{i, l}$ are of the form $(x, l)$ where $x\in W_i\cup \hat W_i$.
The map $f_{\Delta}:\Delta \to \Delta$ is defined at $(x,l) \in \Delta_{i,l}$ as
$$
f_\Delta(x,l) = \left\{ \begin{array}{ll}
(x,l+1) \in \Delta_{i,l+1} & \text{ if } l < \tau_i - 1;   \\
(F(x),0) = (0,f^{S_{i-1}}(x)) \in \sqcup_i\Delta_{i,0} & \text{ if } l = \tau_i - 1.
\end{array} \right.
$$
The projection $\pi:\Delta \to I$, defined by $\pi(x,l) = f^l(x)$ for
$(x,l) \in \Delta_{i,l}$, semiconjugates this map to the original system:
$\pi \circ f_\Delta = f \circ \pi$.
Furthermore, the induced map $(Y,F)$ is isomorphic to the first return
map to the {\em base} $\Delta_0 = \sqcup_i \Delta_{i,0}$.

\begin{lemma}\label{lem:conformal-induced} Let $\Phi_t$ be the induced potential of $\phi_t$,
and $p$ be a potential shift.

\begin{enumerate}[label=({\alph*}),  itemsep=0.0mm, topsep=0.0mm, leftmargin=7mm]
\item A $(\phi_t,p)$-conformal measure $n_t$ for $(I, f)$ yields a
$(\Phi_t,\tau p)$-conformal  measure $m_t$ for $(Y, F)$ by restricting and normalising:
$$
m_t(A) = \frac{1}{n_t(Y)} n_t( A )\quad  \text{ for every } A \subset Y := \cup_{i \ge 1}(W_i \cup \hat W_i).
$$
\item
A $(\Phi_t,\tau p)$-conformal measure $m_t$ for $(Y,F)$
projects to a $(\phi_t,p)$-conformal
measure $n_t$ for $(I, f)$: for every $i,l$
and $A \subset W_i$ or $A\subset \hat W_i$,
$$
n_t(\pi(A,l)) = \frac1{M} \int_A
\exp\left(lp + \sum_{j=0}^{l-1} \phi_t \circ f^j \right) dm_t,
$$
see Figure~\ref{fig:conformal_measure}, with normalising constant
$$
M := 1 + e^p \sum_{i \ge 2} \int_{W_i} e^{-\phi_t} dm_t
+ e^{2p}  \sum_{i \ge 3}\int_{W_i} e^{ - \phi_t \circ f - \phi_t} dm_t \ge 1
$$
is $(\phi_t,p)$-conformal.
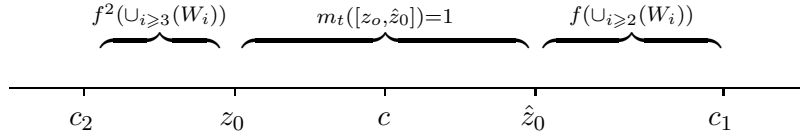
\begin{figure}[ht]
\unitlength=10mm
\begin{picture}(11,2.5)(0,0) \let\ts\textstyle
\put(0,1){\line(1,0){10.5}}
\put(1,0.9){\line(0,1){0.1}} \put(0.8, 0.5){$c_2$}
\put(3,0.9){\line(0,1){0.1}} \put(2.8, 0.5){$z_0$}
\put(5,0.9){\line(0,1){0.1}} \put(4.9, 0.5){$c$}
\put(7,0.9){\line(0,1){0.1}} \put(6.8, 0.5){$\hat z_0$}
\put(9.5,0.9){\line(0,1){0.1}} \put(9.3, 0.5){$c_1$}
\put(1.1,1.5){$\overbrace{\qquad\qquad}^{f^2(\cup_{i \ge 3}(W_i))}$}
\put(3.1,1.5){$\overbrace{\qquad\!\! \qquad\qquad \qquad\qquad}^{m_t([z_o, \hat z_0]) = 1}$}
\put(7.1,1.5){$\overbrace{\qquad\qquad\qquad}^{f(\cup_{i \ge 2}(W_i))}$}
\end{picture}
\caption{Distribution of the conformal mass $n_t$ on $[c_2, c_1]$}
\label{fig:conformal_measure}
\end{figure}

\noindent
In the case that $\phi_t = -t \log|f'|$, then the formula for
the normalising constant simplifies to
$M = 1+ e^p \sum_{i\ge 2} w_i^t \kappa_i^t +
e^{2p} \sum_{i \ge 3} w_i^t \kappa_i^t \kappa_0^t$
which is finite for all $\lambda \in (0,1)$, $t >0$ and $p \in \R$. 
\item
The invariant measure $\mu_t$ for $(Y,F, \Phi_t)$ projects to an invariant
measure $\nu_t$ provided $\sum_i \tau_i \mu_t(W_i \cup \hat W_i) < \infty$
(where in fact $\tau_i = S_{i-1}$),
using the formula
$$
\nu_t = \frac1\Lambda
\sum_i \sum_{j=0}^{\tau_i-1} f^j_* \mu_t \quad \text{ for }
\Lambda = \sum_i \tau_i \mu_t(W_i \cup \hat W_i).
$$
Moreover,
$$h(\nu_t)=\frac{h(\mu_t)}{\Lambda}
\quad  \text{ and } \quad \int g~d\nu_t =\frac{\int G~d\mu_t}{\Lambda},
$$
for any measurable potential $g$ on $I$ and its induced version $G$ on $Y$. 
\end{enumerate}
\end{lemma}

\begin{remark} 
Note that the last part of this lemma is just an application of the \emph{Abramov formula},
see for example \cite[Theorem 2.3]{PS} and \cite[Theorem 5.1]{Zwei_ind}.
\label{rmk:abra}
\end{remark}

\begin{proof}
(a) If $n_t$ is $(\phi_t,p)$-conformal for $(I,f)$, it means,
as stated in Definition~\ref{def:conformal}, that
$n_t(f(A)) = \int_A e^{-\phi_t + p} dn_t$ whenever $f:A \to f(A)$
is one-to-one.
Taking $A \subset W_i$ (or $\subset \hat W_i$), and applying the above
$\tau_i = S_{i-1}$ times gives that
$n_t(F(A)) = \int_A e^{-\Phi_t + \tau_i p} dn_t$, so
the normalised restriction $m_t =  \frac{1}{n_t(Y)} n_t$ is indeed
$(\Phi_t,\tau p)$-conformal.

(b) For the second statement, it is straightforward from the definition
that if $A \subset W_i$ or $A\subset \hat W_i$ and $0 \le l < \tau_i-1$,
then for $B = \pi(A,l)$,
\begin{eqnarray*}
n_t(f(B)) &=& n_t(\pi(A,l+1)) \\
&=&  \frac1M \int_A \exp\left((l+1)p - \sum_{j=0}^l \phi_t \circ f^j \right) dm_t \\
&=& \frac1M \int_A e^{\phi_t \circ f^l + p}
\exp\left(lp - \sum_{j=0}^{l-1} \phi_t \circ f^j \right) dm_t \\
&=& \int_B e^{-\phi_t + p} dn_t.
\end{eqnarray*}
Similarly, if $l = \tau_i-1$, then
\begin{eqnarray*}
n_t(f(B)) &=& n(F(A)) = \frac1{\Lambda} \int_A
\exp\left(\tau_i p - \Phi_t \right) dm_t \\
&=& \frac1M \int_A \exp\left(
\sum_{j=0}^{\tau_i-1} (-\phi_t \circ f^j + p) \right)   dm_t \\
&=& \frac1M \int_A e^{\phi_t \circ f^l + p}
\exp\left(lp - \sum_{j=0}^{l-1} \phi_t \circ f^j \right) dm_t \\
&=& \int_B e^{-\phi_t + p} dn_t
\end{eqnarray*}
This proves the $(\phi_t,p)$-conformality.
The tricky part is to show that $n_t$ is actually
well-defined.
Assume that $B = \pi(A,l) = \pi(A', l')$ for two different sets $A \subset W_i$
and $A' \subset W_{i'}$. So we must show that the procedure above gives $n_t(\pi(A, l))=n_t(\pi(A', l'))$.
Assume also that $\tau_i - l \le \tau_{i'}-l'$; then we might as well take
$B$ maximal with this property:
$B = \pi(W_{i'}, l')$.

Now $B' := f^{\tau_i-l}(B) \subset F(W_i) = (z_{Q(i)}, c)$ or
$(c, \hat z_{Q(i)})$.  It is important to note that
the induced map $F$ is not a first return map to a certain region, but
$F|_{W_i \cup \hat W_i} = f^{S_{i-1}}|_{W_i \cup \hat W_i}$ is the first return
map to $(z_{Q(i)}, \hat z_{Q(i)})$.
Since $f^{\tau_{i'} - \tau_i}$ maps $B'$ to $F(W_{i'}) = (z_{Q(i')}, c)$ or
$(c, \hat z_{Q(i')})$, the iterate  $f^{\tau_{i'} - \tau_i}|_B'$
can be decomposed into an integer number, say $k \ge 0$, of
applications of $F$, and $B'$ is in fact a $k$-cylinder for the induced map.  Since $m_t$ is $(\Phi_t,\tau p)$-conformal,
$$
m_t(B') = \int_{F^k(B')}
\exp\left(\sum_{j=0}^{k-1} (\Phi_t - \tau p) \circ F^{-j} \right) dm_t.
$$
Taking an extra $\tau_i-l$ steps backward, we get 
\begin{eqnarray*}
n_t(B) &=& 
\frac1M \int_{B'}\exp\left(\sum_{j=1}^{\tau_i-l} (\phi_t - p) \circ f^{j-(\tau_i-l)} \right) dm_t \\
 &=& \frac1M \int_{F^k(B')}
\exp\left(\sum_{j=0}^{k-1} (\Phi_t - \tau p) \circ F^{-j} \right)
\exp\left(\sum_{j=1}^{\tau_i-l} (\phi_t - p) \circ f^{j-(\tau_i-l) \circ F^{-k}} \right) dm_t \\
 &=&  \frac1M \int_{F(W_{i'})}
\exp\left(-\sum_{j=1}^{\tau_{i'}-l'} (\phi_t - p) \circ f^{j-(\tau_{i'}-l')} \right) dm_t,
\end{eqnarray*}
so computing $n_t(B)$ using $\tau_i-l$ or $\tau_{i'} - l'$ both give
the same answer.

Now for the normalising constant, since our method of projecting conformal measure only takes the measure of one of the preimages of $\pi$ in $\Delta$ of any set $A\subset I$, we do not sum over all levels of the tower, but just enough so that the image by $\pi$ covers $I$, up to a zero measure set.
However, modulo a countable set, the core $[c_2,c_1]$
is disjointly covered by
$\bigcup_{i \ge 1} (W_i \cup \hat W_i) \cup \bigcup_{i \ge 2} f(W_i)
\cup \bigcup_{i \ge 2} f^2(W_i)$.
This gives
\begin{eqnarray*}
M &=& \sum_{i \ge 1} m_t(W_i \cup W_i) + \sum_{i \ge 2}
\int_{W_i} e^{p-\phi_t} dm_t +   \sum_{i \ge 3}
\int_{W_i} e^{2p-\phi_t-\phi_t \circ f} dm_t \\
 &=& 1 + \sum_{i \ge 2}
e^p \int_{W_i} e^{-\phi_t} dm_t +  e^{2p} \sum_{i \ge 3}
\int_{W_i} e^{-\phi_t-\phi_t \circ f} dm_t.
\end{eqnarray*}
for an arbitrary potential.
Using the formulas for the slope $\kappa_i = f'|W_i$ from \eqref{eq:kappa}
and the expressing for $\frac12 w_i = m_t(W_i)$ from \eqref{eq:m_t_W},
we obtain for $\phi_t = -t \log|f'|$:
\begin{eqnarray*}
M &=& 1 + \frac{e^p}{2} \sum_{i \ge 2} w_i^t \kappa_i^t +
\frac{e^{2p}}{2} \sum_{i \ge 3}
w_i \kappa_i^t \kappa_0^t \\
 &=& 1 + \frac{e^p (1-\lambda^t)}{2}\left(1  + (1-\lambda)^t \lambda^t
+ (1-\lambda)^{3t} \lambda^{2t}
+ \sum_{i \ge 5} \frac{\lambda^{3ti}(1-\lambda)^{2ti}}{\lambda^{11t}(1-\lambda)^{5t}} \right) \\
&& + \ \frac{e^{2p} (1-\lambda^t)}{2(1-\lambda)^t}
\left( (1-\lambda)^t \lambda^t
+ (1-\lambda)^{3t} \lambda^{2t}
+ \sum_{i \ge 5} \frac{\lambda^{3ti}(1-\lambda)^{2ti}}{\lambda^{11t}(1-\lambda)^{5t}} \right)  \\
 &=& 1 + \frac{e^p (1-\lambda^t)}{2}\left(1  + (1-\lambda)^t \lambda^t
+ (1-\lambda)^{3t} \lambda^{2t} + \frac{\lambda^{4t}(1-\lambda)^{5t}}{1-\lambda^{3t}(1-\lambda)^{3t}} \right) \\
&& + \ \frac{e^{2p} (1-\lambda^t) \lambda^t}{2}
\left( 1
+ (1-\lambda)^{2t} \lambda^{t}
+ \frac{\lambda^{3t}(1-\lambda)^{4t}}{1-\lambda^{3t}(1-\lambda)^{3t}} \right) < \infty.
\end{eqnarray*}
(c) The third statement is an Abramov formula, see Remark~\ref{rmk:abra}.
\end{proof}

\section{The conformal measure and equilibrium state for $(Y, F_\lambda, \Phi_t)$}
\label{sec:conformal}

In this section we adapt the results for the map $T_\lambda$ studied in \cite{BT3} to the map $F_\lambda$. This also allows us to prove Theorem~\ref{thm:main hyp dim}.  

\begin{proposition}\label{prop:conformal}
For each $\lambda \in (0,1)$, $t > 0$ and $p = P(\Phi_t)$, the map 
$F_\lambda$ has a  $(\Phi_t-p)$-conformal measure $\tilde m_t$ with
\begin{equation}\label{eq:m_t_W}
 \tilde m_t(W_j) = \tilde m_t(\hat W_j) =
\left\{ \begin{array}{ll}
\frac{1-\lambda^t}{2}\lambda^{t(k-1)} & \text{ if } \lambda^t \le \frac12, \\[1mm]
\left[ (k-1) + \lambda^{-t}(1-\frac{k}{2}) \right] (\frac12)^{k+1}
& \text{ if } \lambda^t \ge \frac12.
\end{array} \right.
\end{equation}
If in addition $\lambda^t < \frac12$, then $F_\lambda$ preserves a 
probability measure $\tilde \mu_t \ll \tilde m_t$ with
\begin{equation}\label{eq:mu_t_W}
 \tilde \mu_t(W_j) = \zeta_t \frac{1-2\lambda^t}{\lambda^t} \left( \frac{\lambda^t}{1-\lambda^t} \right)^j
\quad \text{ and  } \quad
 \tilde \mu_t(\hat W_j) = (1-\zeta_t) \frac{1-2\lambda^t}{\lambda^t} \left( \frac{\lambda^t}{1-\lambda^t} \right)^j
\end{equation}
for some $\zeta_t \in (0,1)$.
Moreover, $ \tilde \mu_t$ is an equilibrium state for potential $\Phi_t$.
\end{proposition}

\begin{proof}
Recall from \eqref{eq:StrVogt} and \eqref{eq:comm_diag}
that $T_\lambda \circ \pi = \pi \circ F_\lambda$ for the
two-to-one factor map $\pi$ with $\pi^{-1}(V_j) = W_j \cup \hat W_j$.
In \cite[Theorem 2]{BT3} it is shown that $T_\lambda$ has a 
$(\Phi_t-p)$-conformal measure such that for $\psi(t) := \frac{(1-\lambda)^t}{1-\lambda^t}$,
$$
m_{t,p}(V_k)= \left\{ \begin{array}{ll}
(1-\lambda^t)\lambda^{t(k-1)} & \text{ if } p = \log \psi(t) \text{ and } \lambda^t \le \frac12, \\[1mm]
\left[ (k-1) + \lambda^{-t}(1-\frac{k}{2}) \right] (\frac12)^k
& \text{ if } p = \log 4[\lambda(1-\lambda)]^t \text{ and } \lambda^t \ge \frac12.
\end{array} \right.
$$
and an invariant measure (provided $\lambda^t < \frac12$) with
$\mu_{t,p}(V_k) = \frac{1-2\lambda^t}{\lambda^t} 
\left( \frac{\lambda^t}{1-\lambda^t} \right)^k$.
To obtain $ \tilde m_t$ and $ \tilde \mu_t$ we lift these measures by $\pi$,
distributing the mass to $W_j$ and $\hat W_j$ appropriately.
Since $F_\lambda(W_j) = F_\lambda(\hat W_j) = \cup_{k \ge j-1} W_k$
or $\cup_{k \ge j-1} \hat W_k$, we can distribute the conformal mass
evenly. This gives \eqref{eq:m_t_W}.

To obtain \eqref{eq:mu_t_W}, first define
\begin{equation}\label{eq:At}
A^t = (1-\lambda)^t
\left( \begin{array}{ccccccc}
1^t    & \lambda^t & \lambda^{2t} & \lambda^{3t} & \hdots & \hdots   & \hdots \\
1^t    & \lambda^t & \lambda^{2t} & \lambda^{3t} &        &          &   \\
0      & 1^t       & \lambda^t   & \lambda^{2t} & \lambda^{3t}   &   &   \\
0      & 0         & 1^t         & \lambda^t   & \lambda^{2t} &     &  \\
\vdots &           & 0           & 1^t         & \lambda^t   & \lambda^{2t} & \dots \\
       &           &             &             & \ddots      & \ddots    & \ddots
\end{array} \right).
\end{equation}
That is, the matrix $A$ in \eqref{eq:A} with all entries raised to the power $t$,
then  $\psi^{-1}(t) A^t$ is a probability matrix and
$\frac{m_{t,p}(V_i\cap T_\lambda^{-1}(V_j)) }{ m_{t,p}(V_i) } =
\psi^{-1}(t) A^t_{i,j}$.

Now set  $v^t_j = \frac{1-2\lambda^t}{\lambda^t} \left( \frac{\lambda^t}{1-\lambda^t} \right)^j$ (so for $t = 1$, this reduces to the value of $v_j$ in
\eqref{eq:v}) and define
$$
\zeta_t = \sum_{\stackrel{j \ge 1}{c_{S_{j-1}} < c}}  v^t_j,
$$
\ie the proportion of the invariant mass that maps under $F_\lambda$
to the left of $c$.
Next define $\mu_t$ on cylinders  $C_{i_0\cdots i_{n-1}}$ by 
$$
\mu_t(C_{i_0\cdots i_{n-1}}) = \zeta_t v^t_{i_0} \prod_{k=1}^{n-1}
\psi(t)^{-1} A^t_{i_{k-1}i_k}
$$
and similarly
$\mu_t(\hat C_{i_0\cdots i_{n-1}}) =
 \hat \zeta_t v^t_{i_0} \prod_{k=1}^{n-1}
\psi(t)^{-1} A^t_{i_{k-1}i_k}$.

Since $\psi(t)^{-1} A^t$ is a probability matrix with $A^t_{ki_0} = 0$ if $k > i_0+1$, we get
for every cylinder set
\begin{eqnarray*}
 \tilde \mu_t(F_\lambda^{-1}(C_{i_0\cdots i_{n-1}})) &=& \sum_{\stackrel {k \le i_0+1}{c_{S_{k-1}} < c} }
 \tilde \mu_t(C \cup \hat C)_{ki_0\cdots i_{n-1}} \\
&=&\zeta_t \sum_{k \le i_0+1} v^t_k \psi(t)^{-1} A^t_{k i_0}
 \prod_{j=1}^{n-1} \psi(t)^{-1} A^t_{i_{j-1}e_j} \\
&=& \zeta_t v^t_{i_0} \prod_{k=1}^{n-1} \psi(t)^{-1} A^t_{i_{k-1}i_k}
=  \tilde \mu_t(C_{i_0\cdots i_{n-1}})
\end{eqnarray*}
and similarly for $F_\lambda^{-1}(\hat C_{e_0\dots e_{n-1}})$. 
This proves $F_\lambda$-invariance of $ \tilde \mu_t$.

The $T_\lambda$-invariant measure above is the unique equilibrium state
for $-t \log|T'_\lambda|$ provided $\lambda^t < \frac12$.
Since the factor map $\pi$ does not affect entropy, and
because for any $F_\lambda$-invariant measure $\tilde \nu$ we have
$\int \log |F'_\lambda| d\tilde \nu = \int \log |T'_\lambda| d(\tilde \nu \circ \pi^{-1})$,
it follows that $\tilde \mu_t$ is indeed the unique equilibrium state for
$(Y, F_\lambda, -t \log|F'_\lambda|)$.
\end{proof}

\begin{proof}[Proof of Theorem~\ref{thm:main hyp dim}]
Let $x \in [z_0, \hat z_0] \setminus \cup_{n \ge 0} f^{-n}(c)$ be arbitrary. 
Since $z_k$ is a closest precritical point,
$f^j(W_k \cup \hat W_k) \cap [z_k, \hat z_k] = \emptyset$ if $0 < j < S_k$.
Therefore, if $c \in \omega(x)$ then $F^i(x) \to 0$ along a subsequence.
From this we see that hyperbolic sets for $F$ coincide with intersections
of hyperbolic sets for $f$ with $[z_0, \hat z_0]$, 
implying that hyperbolic dimension are the same for $F$ and $f$.

Now for the escaping set, first observe that
the intervals $f^j([z_k, c]) = f^j([c, \hat z_k])$ for $0 < j \le S_k$
have $f^j(c)$ as boundary point and lengths tending to $0$ as
$k \to \infty$.
Therefore $F^i(x) \to c$ implies that $f^n(x) \to \omega(c)$ which implies that $\omega(x)= \omega(c)$.  We next show that $F^i(x) \to c$ if and only if  $\omega(x)= \omega(c)$.

Denote by $U_n$ the largest neighbourhood of $x$ on which $f^n$ is 
monotone, and let $R_N$ be the largest distance
between $f^n(x)$ and $\partial f^n(U_n)$.
If there is $k$ such that $F^i(x) \notin [z_k, \hat z_k]$
infinitely often, then  by the Markov property of $F$, 
$f^n(U_n) \supset [z_k, c]$
or $[c, \hat z_k]$ along a subsequence.
This means $R_n \not \to 0$.
By \cite{Btams}, this implies that $\omega(x) \not\subset \omega(c)$.

Therefore $\omega(x) = \omega(c)$ if and only if $F^i(x) \to c$,
and hence the escaping set $\Omega_\lambda$ coincides with 
$\Bas_\lambda \cap [z_0, \hat z_0]$.
Theorem~\ref{thm:main hyp dim} therefore follows from \cite[Theorem C]{BT3}. 
\end{proof}

\section{Conformal pressure for $(I,f,\phi_t)$} \label{sec:conformal_for_f}

In this section we prove the main part of Theorem~\ref{mainthm:thermo_Fibo}, with the components about 
existence of conformal measure and upper and lower bounds on conformal pressure in various lemmas.
We start by giving the definition of conformal pressure, presented for general dynamical systems.

\begin{definition}
For a dynamical system $g:X\to X$ and a potential $\phi:X\to [-\infty, \infty]$, the {\em conformal pressure} for $(X,g,\phi)$ is
\begin{equation*}\label{eq:Pconf}
\Pconf(\phi):=\inf\left\{p\in \R:\text{there exists a } (\phi- p)
\text{-conformal measure}\right\}.
\end{equation*}
\end{definition}

The results on the pressure in this section are obtained using 
$\Pconf(\phi_t)$; in Section~\ref{sec:CMS} we show that the conformal pressure $\Pconf(\phi_t)$  
coincides with the (variational) pressure $P(\phi_t)$ from \eqref{eq:pressure}. Thus our statements in Theorem~\ref{mainthm:thermo_Fibo} should be read as applying to `both' quantities.
For $\Pconf (\Phi_t)$, we start by quoting the conclusion of 
Theorems~2 and B of \cite{BT3}: $\Pconf(\Phi_t)$ and $P(\Phi_t)$
coincide, and
\begin{equation*}\label{eq:PconfPhit}
\Pconf(\Phi_t) =  \left\{
\begin{array}{ll}
\log \psi(t)  & \text{ if } \lambda^t \le \frac12;\\[2mm]
\log[4\lambda^t(1-\lambda)^t] \qquad & \text{ if } \lambda^t \ge \frac12.
\end{array} \right.
\end{equation*}
Recall from  \eqref{eq:t1} that $t_2 = -\log 4/\log[\lambda(1-\lambda)]$
is the value of $t$ such that $[\lambda(1-\lambda)]^t = \frac14$.
Hence $t_2 = t_1$ if $\lambda \ge \frac12$ and $t_2 < t_1 = 1$ otherwise.
We can interpret $t_1$ as the smallest $t$ such 
that the pressure of the induced system $\Pconf(\Phi_t) = 0$.

Any $(\Phi_t-p\tau)$-conformal measure $m_t$ must
observe the relations (for $\w_k = m_t(W_k) = m_t(\hat W_k)$)
\begin{eqnarray}
\w_1 &=& (1-\lambda)^t e^{-pS_0}\nonumber \\
\w_2 &=& \lambda^t (1-\lambda)^t e^{-pS_1}\nonumber \\
\w_3 &=& \lambda^t (1-\lambda)^t e^{-pS_2} (1-\w_1)
\label{eq:tildew} \\
\vdots\ && \qquad\vdots \qquad\qquad \vdots \nonumber\\
\w_j &=& \lambda^t (1-\lambda)^t e^{-pS_{j-1}} \left(1-\sum_{k \le
j-2} \w_k\right). \nonumber
\end{eqnarray}

Recurrence relations of a similar form were used in \cite{BT3} to prove \cite[Theorem 2]{BT3}, 
but our situation here is more complicated since in that
setting in the place of each $e^{-pS_j}$ term was simply the constant
term $\psi(t)$. 
The idea now is to find a solution $p = p(t)$ of
\eqref{eq:tildew} such that also
$H(p,t) := \sum_j \w_j$ is equal to $1$ (this is equivalent to finding a solution set $\{ \w_j\}_j$). 
Note that in Lemma~\ref{lem:lowerboundp0} and Proposition~\ref{prop:upperboundp0} below, 
we give necessary lower and upper bounds on $p(t)$, without assuming the existence of a solution.
Along the way, we will also need to check that
$\w_k > 0$ for all $k \ge 1$.

Write
\begin{equation*}%\label{eq:beta}
\beta := t \log [\lambda(1-\lambda)] \quad \text{ and } \quad
\beta' := (t-t_2) \log [\lambda(1-\lambda)],
\end{equation*}
so that $e^\beta = \frac14$ for $t = t_2$ and
$e^\beta = \frac14 e^{\beta'} > \frac14$ for $t < t_2$.

\subsection{Lower bounds on $\Pconf(\phi_t)$}
Let us now compute the asymptotics of $\w_k$ 
to show that in this case $p(t)$ has to be positive for $t < t_1$.

\begin{lemma}\label{lem:tildewp=0}
Fixing $p=0$, there is a unique solution to \eqref{eq:tildew}, denoted by 
$(\bar w^t_k)_{k \in \N}$.
It satisfies 
$$
\begin{cases}
\bar w_k^t > 0 \text{ and } \sum_k \bar w_k^t = 1 & \text{ if } t \ge t_1; \\[1mm]
\text{there exists } k_0 \text{ such that } \bar w_{k_0}^t < 0
 & \text{ if } t < t_1.
\end{cases}
$$
Moreover, for $r_\pm = \frac12(1\pm\sqrt{1-4e^\beta})$, 
$$
\begin{cases}
k_0  =  \Bigg\lceil \frac{\log \frac{r_+(r_+-\lambda^t)}{r_-(r_--\lambda^t)}}
{\log\frac{r_+}{r_-}} \Bigg\rceil + 1
 & \text{ if }\ t_2 <t < t_1 = 1 \text{ (\ie} \lambda \in (0, \frac12)); \\[4mm]
k_0 \approx \frac{2(1-\lambda^t)}{1-2\lambda^t} &
 \text{ if } t \lesssim t_2  \le t_1 =1  \text{ for } \lambda \in (0,\frac12],\\[4mm]
k_0 \approx  \frac{2\pi}{\sqrt{\beta'}}
 & \text{ if } t \lesssim t_1 \le 1 \text{ for } \lambda \in [\frac12,1).
\end{cases}
$$
\end{lemma}

\begin{proof}
Subtracting two successive equations in \eqref{eq:tildew}, we find that
the $\bar w_k^t$ satisfy recursive relation
\begin{equation*}%\label{eq:tildew_difference0}
\bar w^t_{k+1} = \bar w^t_k - e^\beta \bar w^k_{k-1}.
\end{equation*}
The roots of the corresponding generating equation 
$r^2-r+e^\beta=0$ are $r_\pm = \frac{1\pm\sqrt{1-4e^{\beta}}}2$. 
It is straightforward to check that
$$
\begin{cases}
(i)\ r_- < \lambda^t, 1-\lambda^t < r_+ & \text{ if } t > 1; \\[1mm]
(ii)\ r_\pm \in \{\lambda, 1-\lambda\} & \text{ if } t = 1; \\[1mm]
(iii)\ \lambda^t < r_- < r_+ <  1-\lambda^t & \text{ if } t_2 < t < 1
\text{ and } \lambda^t < \frac12; \\[1mm]
(iv)\ 1-\lambda^t < r_- < r_+ <  \lambda^t & \text{ if } t_1 < t < 1
\text{ and } \lambda^t > \frac12;  \\[1mm]
(v)\ r_- = r_+ = \frac12   & \text{ if } 
\begin{cases}
t = t_1 < 1 & \text{ for } \lambda^t > \frac12,\\
t = t_2  & \text{ for } \lambda^t \le \frac12,
\end{cases} 
\\[5mm]
(vi)\ r_\pm \text{ are complex conjugate} &
 \text{ if }
 \begin{cases}
t < t_1 < 1 & \text{ for } \lambda^t > \frac12,\\
t < t_2  \le t_1 =1 & \text{ for } \lambda^t \le \frac12.
\end{cases} 
\end{cases}
$$
(i)-(iv)\ In the first four cases, \ie $r_\pm$ are real and distinct,
the recursion combined with the initial values
$\bar w^t_1 = (1-\lambda)^t$ and $\bar w^t_2 = \lambda^t (1-\lambda)^t$,
give the solution
\begin{equation}\label{eq:mt}
\bar w_k^t = \frac{(1-\lambda)^t}{\sqrt{1-4e^\beta}}
\left[ (\lambda^t-r_-) r_+^{k-1} +
(r_+-\lambda^t) r_-^{k-1} \right].
\end{equation}
If $t \ge 1$, then the coefficients are non-negative, and also
if $t_1 < t < 1$.
If $t_2 < t < t_1 = 1$,
then the coefficient $\lambda^t-r_- < 0$, so there is $k_0$
such that $\bar w_k^t < 0$ for all $k \ge k_0$, 
namely 
\begin{equation}\label{eq:k0rpm}
 \frac{r_+^2}{r_-^2}\frac{r_+-\lambda^t}{r_--\lambda^t} \ge
\left(\frac{r_+}{r_-}\right)^{k_0} > \frac{r_+}{r_-}\frac{r_+-\lambda^t}{r_--\lambda^t},
\end{equation}
which results in 
$k_0 = \left\lceil \frac{\log \frac{r_+(r_+-\lambda^t)}{r_-(r_--\lambda^t)}}
{\log\frac{r_+}{r_-}}\right\rceil+1$. 

(v)\ If $t = t_1 < 1$, or when $t=t_2$,
then $r_- = r_+ = \frac12$, and the general solution is
$$
\bar w_k^t = \frac{(1-\lambda)^t}{2^k} \ \left( 4(1-\lambda^t) + 2k(2\lambda^t-1) \right).
$$
If $\lambda^t \ge \frac12$ (\ie $t = t_1 \le 1$), then 
the coefficient $4(1-\lambda^t) + 2k(2\lambda^t-1) > 0$
and hence $\bar w_k^t > 0$ for all $k$.
If $\lambda^t < \frac12$, then the coefficient
$4(1-\lambda^t) + 2k(2\lambda^t-1) < 0$
for all $k \ge k_0 = \left \lceil 2(1-\lambda^t)/(1-2\lambda^t) \right\rceil+1$.

(vi)\ Finally, if $t < t_1 < 1$,  or in general when $t < t_2$,
then the roots are complex. Together with the initial values
$\bar w^t_1 = (1-\lambda)^t$ and $\bar w^t_2 = \lambda^t (1-\lambda)^t$,
we find the solution
\begin{eqnarray}\label{eq:tildewt<t0}
\bar w_k^t &=& \frac{(1-\lambda)^t}{2^k} \left[
\left( 4 \cos \frac{\sqrt{\beta'}}{2} - 4\lambda^t\right) \cos \left(\frac{\sqrt{\beta'}k}{2}\right)  + \right. \nonumber \\
&& \qquad \left.  \left( 4\lambda^t \cos\left(\frac{\sqrt{\beta'}}{2}\right) - 2 \cos\left( \sqrt{\beta'}\right) \right)
\left(\frac{\sin\frac{\sqrt{\beta'}k}{2}}{\sin\frac{\sqrt{\beta'}}{2}}\right) \right] \nonumber \\
&=& \frac{2 (1-\lambda)^t}{2^k} \left[
( 2 \cos \theta - 2\lambda^t ) \cos \theta k  
+ (2\lambda^t \cos \theta - 2 \cos^2 \theta + 1)
\frac{\sin \theta k}{\sin\theta} 
\right], \quad 
\end{eqnarray}
for $\theta = \frac12 \sqrt{\beta'} = \frac12 \sqrt{(t-t_2) \log[\lambda(1-\lambda)]}$.
This is an oscillatory function in $k$, with an exponential decreasing 
coefficient $2^{-k}$.
Recall that $\bar w_2^t =  \lambda^t \bar w_1^t > 0$.
First assume that $\lambda \ge \frac12$, whence 
$\lambda^t > \frac12$. Therefore 
$$
0 < 2\cos \theta - 2\lambda^t \ll
\frac{2\lambda^t \cos \theta - \cos^2 \theta + 1}{\sin \theta},
$$ 
so the expression in the square brackets
becomes negative when 
$k_0 \approx \frac{2\pi}{\sqrt{\beta'}}$.

Now set $\lambda < \frac12$, and $\lambda^t < \frac12$ and moreover assume $t-t_2$ is small.
Then $2\lambda^t \cos \theta - 2 \cos^2 \theta + 1 < 0$, so approximating
$\sin \theta k/\sin \theta = k$ for small values of $\theta$, we find
 the expression in the square brackets
becomes negative when 
$k_0 > \frac{2\cos \theta - 2\lambda^t}{2 \cos^2 \theta - 1 - 2\lambda^t \cos \theta} \approx \frac{2(1-\lambda^t)}{1-2\lambda^t}$.

Note also that in all cases $\bar w_k^t \to 0$, and therefore
\eqref{eq:tildew} gives that $1-\sum_{k<j-1} \bar w_k^t \to 0$
as $j \to \infty$. This shows that $\sum_k\bar w_k^t = 1$.
\end{proof}

We can now use Lemma~\ref{lem:tildewp=0} to address directly the problem set up in \eqref{eq:tildew}: finding a solution $p=p(t)$ to $H(p,t)=1$ with all summands non-negative.

\begin{lemma}\label{lem:lowerboundp0} 
If $\lambda \ge \frac12$ and $t < t_1 \le 1$ is close to $t_1$, or if
$\lambda < \frac12$ and $\lambda^t$ is sufficiently close to $\frac12$,
then there is $\tau_0  = \tau_0(\lambda) > 0$ such that 
$p(t) > \frac{\tau_0}{S_{k_0}}$.
\end{lemma}

\begin{proof}
Let $\bar w_k^t$ be the solution of \eqref{eq:tildew} for $p=0$ as computed in 
Lemma~\ref{lem:tildewp=0}, while we write $\w_k=\w_k(p)$ for the
case $p > 0$. 
We start by showing that, under the assumptions of the lemma, 
$\bar w^t_{k+1}/\bar w^t_k \approx \frac12$ for $1 \le k \le k_0-10$.\\
{$\bullet$ Case 1:}  $\lambda \ge \frac12$ and $t < t_1 \le 1$ is close to $t_1$.
In this case,
$2\lambda^t \cos \theta - 2 \cos^2 \theta + 1 =
(2\lambda^t-1)\cos \theta + (1+2\cos \theta)(1-\cos \theta) > 0$
for $0 \le \theta = \frac12 \sqrt{\beta'} \le \pi/2$.
With $\bar w_k^t$ as given by \eqref{eq:tildewt<t0} and
using standard trigonometric formulas, we derive that
\begin{eqnarray*}%\label{eq:wk-quotient}
\frac{\bar w_{k+1}^t}{\bar w_k^t}
&=& \frac12 \left(\cos \theta -
\sin \theta \frac{ (2 \cos \theta - 2\lambda^t) \sin \theta k
- \frac{2\lambda^t \cos \theta - 2 \cos^2 \theta + 1}{\sin \theta} \cos \theta k}
{  (2 \cos \theta - 2\lambda^t) \cos \theta k +
 \frac{2\lambda^t \cos \theta - 2 \cos^2 \theta + 1}{\sin \theta} \sin \theta k} \right) \nonumber \\
&\sim& \frac12 \left(\cos \theta + \frac{\sin \theta}{\tan \theta k}
 \right) \qquad \text{ as } \theta \to 0.
\end{eqnarray*}
If $10 \le k \le k_0-10$, this reduces to 
\begin{equation*}%\label{eq:wk-quotient2}
\frac{11}{20} \ge \frac{\bar w_{k+1}^t}{\bar w_k^t} = \frac12\left(\cos \theta - \frac{\sin \theta}{\tan \theta k}\right)  
\ge \frac{9}{20} \quad \text{ for small } \theta.
\end{equation*}
$\bullet$ Case 2: $\lambda < \frac12$ and $\lambda^t$ is sufficiently close to $\frac12$.
In this case $\bar w_k^t$ is given by \eqref{eq:mt}, so
\begin{eqnarray*}
\frac{\bar w_{k+1}^t}{\bar w_k^t} &=& \frac{(\lambda^t-r_-) r_+^k + (r_+-\lambda^t) r_-^k}{(\lambda^t-r_-) r_+^{k-1} + (r_+-\lambda^t) r_-^{k-1}} \\
&=& r_+ \cdot \frac{1+ \frac{r_+-\lambda^t}{\lambda^t-r_-} \left( \frac{r_-}{r_+} \right)^k}
{1+ \frac{r_+-\lambda^t}{\lambda^t-r_-} \left( \frac{r_-}{r_+} \right)^{k-1} }
=  r_+ \cdot \frac{1+ \frac{r_+-\lambda^t}{\lambda^t-r_-} \left( \frac{r_-}{r_+} \right)^{k_0} \left( \frac{r_+}{r_-} \right)^{k_0-k} }
{1+ \frac{r_+-\lambda^t}{\lambda^t-r_-} \left( \frac{r_-}{r_+} \right)^{k_0}  \left( \frac{r_+}{r_-} \right)^{k_0-k-1} }.
\end{eqnarray*}
Using \eqref{eq:k0rpm}, we obtain
$$
r_+ \le \frac{\bar w_{k+1}^t}{\bar w_k^t} \le r_+ \cdot \frac{1 + \left( \frac{r_+}{r_-} \right)^{k_0-k+2} }
{1 + \left( \frac{r_+}{r_-} \right)^{k_0-k} } \le r_+ \left( \frac{r_+}{r_-} \right)^2.
$$
Since $r_+, r_- \to \frac12$ as $\lambda^t \to \frac12$, we obtain that  
$\frac{\bar w_{k+1}^t}{\bar w_k^t} \approx \frac12$ uniformly in $k$ in this case.

The difference between $\bar w_k^t$ and $\w_k$ is 
$\eps_k =\eps_k(p)= \w_k (p)- \bar w_k^t$.
We claim that if $p < 1/S_{k_0}$, then there is $K$ such that 
\begin{equation}\label{eq:epsk}
|\w_k - \bar w^t_k| =: |\eps_k| \le K\tau_0(1-e^{-pS_{k-1}}) \bar w_k^t \quad \text{ for all } k \le k_0-10.
\end{equation}
Since $\bar w_1^t (e^{-p}-1) = \eps_1 \le -pS_0 \bar w_1^t$
and  $\bar w_2^t (e^{-2p}-1) = \eps_2 \le -pS_1 \bar w_2^t$,
this claim holds for $k = 1,2$. 

Subtracting two successive equations
in \eqref{eq:tildew} gives the recursive relations
\begin{equation}\label{eq:wkrecursive}
\w_{k+1} = e^{-pS_{k-2}} \w_k - e^{\beta-pS_k} \w_{k-1},
\end{equation}
so for $p = 0$ this is $\bar w_{k+1}^t = \bar w_k^t - e^{\beta} \bar w_{k-1}^t$.
For $\eps_k$ we obtain
\begin{eqnarray*}
\eps_{k+1} &=&  \w_{k+1} - \bar w^t_{k+1} \\
&=& e^{-pS_{k-2}} \eps_k - e^\beta e^{-pS_k} \eps_{k-1} + e^\beta(1-e^{-pS_k})\bar w^t_{k-1} 
-(1-e^{-pS_{k-2}}) \bar w^t_k.
\end{eqnarray*}
Write $\eps_k = u_k (1-e^{-pS_{k-1}}) \bar w_k^t$, so $u_1 = u_2 = -1$
and $u_3 \in (-1,0)$.
Then we can rewrite the above as
\begin{eqnarray*}
u_{k+1} &=& e^{-pS_{k-2}} \frac{1-e^{-pS_{k-1}}}{1-e^{-pS_k}} \frac{\bar w_k^t}{\bar w_{k+1}^t} u_k
- e^\beta e^{-pS_k}\frac{1-e^{-pS_{k-2}}}{1-e^{-pS_k}}  \frac{\bar w_{k-1}^t}{\bar w_{k+1}^t} u_{k-1} \\
&& \qquad
+\ \frac{\bar w_k^t}{\bar w_{k+1}^t} \left(e^\beta \frac{\bar w_{k-1}^t}{\bar w_k^t} - \frac{1-e^{-pS_{k-2}}}{1-e^{-pS_k}}\right) \\
&:=& au_k - bu_{k-1} + c.
\end{eqnarray*}
The numbers $a,b,c$ depend on $k$, but since
$\frac{\bar w_{k+1}^t}{\bar w_k^t} \approx \frac{\bar w_k^t}{\bar w_{k-1}^t} \in [0.45, 0.55]$ for all $10 \le k \le k_0-10$,
and $e^\beta \approx \frac14$,
we have $c \in [0.1, 0.5]$ and $0 < a-b < 0.99$.
Therefore the orbit $(u_k)_{k \ge 1}$ is bounded,
say $|u_k| \le K$ for all $k$, and in fact positive from the moment that two
consecutive terms are positive.
In particular, $-1 \le u_k \le K$ for all $k$, and 
$|\eps_k| \le K(1-e^{-pS_{k-1}}) \bar w_k^t$ for all
$k \le k_0-10$, proving Claim~\eqref{eq:epsk}.
If we now take $p \le \tau_0/S_{k_0}$,
then $|\eps_k| \le K \tau_0 \gamma^{-11} \bar w_k^t$ for $k = k_0-10$.
Propagating this tiny error (provided $\tau_0$ is small)
for another eleven iterates, \ie eleven recursive steps
$\w_{k+1} = e^{-pS_{k-2}} \w_k - e^{\beta-pS_k} \w_{k-1}$,
we find that $\w_{k_0+1} < 0$.
This shows that $p(t) > \tau_0/S_{k_0}$.
\end{proof}

Recall that $\gamma = \frac12(1+\sqrt{5})$ 
and $\Gamma  = \frac{2 \log \gamma}{\sqrt{-\log[\lambda(1-\lambda]}}$. 

\begin{proposition}\label{prop:lowerboundp0} 
There are $\tau_0  = \tau_0(\lambda)$ and $\tilde C = \tilde C(\lambda) > 0$ such that 
$$
p(t) > \frac{\tau_0}{S_{k_0}}
\ge 
\begin{cases}
\tau_0  e^{-\pi\Gamma/\sqrt{t_1-t}} & \text{ if } t < t_1 \le 1 \text{ close to $t_1$ and } \lambda \ge \frac12;\\
\tau_0 \tilde C (1-t)^{\frac{\log(\gamma)}{\log R} } & \text{ if } t < 1 \text{ close to $1$  and } \lambda < \frac12,
\end{cases} 
$$
where $\log R = 2\log(1+\sqrt{1-4\lambda^t(1-\lambda)^t}) - \log[4\lambda^t(1-\lambda)^t] \sim 2(1-2\lambda)$ as $t \to 1$ and $\lambda \to \frac12$.
\end{proposition}

\begin{proof}
Lemma~\ref{lem:lowerboundp0} gives $p(t) > \frac{\tau_0}{S_{k_0}}$.
For the second inequality, first assume that 
$\lambda \ge \frac12$ and $t < t_1 \le 1$.
Using the estimate of $k_0$ from Lemma~\ref{lem:tildewp=0},
and $\beta' = \sqrt{-\log[\lambda(1-\lambda)] (t_1-t)}$, we find 
\begin{equation*}%\label{eq:eqpt}
p(t) \ge \frac{\tau_0}{S_{k_0}} \approx \tau_0 e^{-k_0 \log \gamma}
\ge \tau_0 e^{-\frac{\pi \Gamma}{\sqrt{t_1-t}} }.
\end{equation*}
Now for the case $\lambda < \frac12$ and $t < 1$,
recall from \eqref{eq:k0rpm} that
$$
\frac{\tau_0}{S_{k_0}} \ge
\tau_0 \left( \frac{r_+}{r_-} \right)^{- k_0 \frac{\log(\gamma)}{\log(\frac{r_+}{r_-})}}
\ge \left( \frac{r_+^2}{r_-^2}  \cdot \frac{r_+-\lambda^t}{r_--\lambda^t}\right)^{- \frac{\log(\gamma)}{\log(\frac{r_+}{r_-})}}.
$$
We work out the asymptotics for fixed $\lambda < \frac12$ and 
first order Taylor expansions for $t \approx 1$.
\begin{eqnarray*}
4e^\beta &=& 4\lambda(1-\lambda)\left(1+ \log[\lambda(1-\lambda)] (t-1) \right)  + \hot \\
\sqrt{1-4e^\beta} &=& (1-2\lambda) \sqrt{1-\frac{4\lambda(1-\lambda)}{(1-2\lambda)^2} \log[\lambda(1-\lambda)] (t-1) + \hot} \\
&=&   (1-2\lambda) \left(1-\frac{2\lambda(1-\lambda)}{(1-2\lambda)^2} \log[\lambda(1-\lambda)] (t-1) \right) + \hot\\
R := \frac{r_+}{r_-} &=& \frac{(1+\sqrt{1-4e^\beta})^2}{4e^\beta} = 1+2(1-2\lambda) + \hot \\
r_+ - \lambda^t &=& 1-2\lambda + \hot \\
r_--\lambda^t  &=& \left( \frac{2\lambda(1-\lambda)}{1-2\lambda} \log[\lambda(1-\lambda)] - 2\lambda \log \lambda \right) (t-1) + \hot
\end{eqnarray*}
This gives exponent $\log(\gamma)/\log(R)$ (which is $\ \sim 
\log(\gamma)/(2(1-2\lambda))$ as $\lambda \to \frac12$) and
$$
\frac{r_+^2}{r_-^2}  \cdot \frac{r_+-\lambda^t}{r_--\lambda^t}
= \frac{(1+4(1-2\lambda)) \left(\frac{1-2\lambda}
{ -\lambda(1-\lambda) \log[\lambda(1-\lambda)] + 2\lambda(1-2\lambda) \log \lambda}\right)
}{1-t}  + \hot
$$
Hence the estimate holds for $0 < \tilde C \sim
\left( \frac{ -\lambda(1-\lambda) \log[\lambda(1-\lambda)] + 2\lambda(1-2\lambda) \log \lambda}{(1+4(1-2\lambda))(1-2\lambda)} \right)^{\frac{\log \gamma}{2\log(1-2\lambda)}}
$ as $\lambda \to \frac12$.
\end{proof}

\begin{lemma}\label{lem:tildewt<t0p>0}
If $p > 0$, then $\w_k \to 0$ super-exponentially:
\begin{equation}\label{eq:wkalpha}
\w_k = 
\begin{cases}
e^{\beta k - p S_{k+1} + \alpha_k} & \text{ if } %H(p,t)=1,\\
 \sum_k \w_k = 1, \\
e^{\beta  - p S_{k-1} + \alpha_k}& \text{ otherwise, }
\end{cases}
\end{equation}
where  $(\alpha_k)_{k \ge 1}$ is a convergent
sequence depending on $p$ and $t$.
\end{lemma}

\begin{proof}
First note that if $\sum_k \w_k \neq 1$, then the factor
$e^{-pS_{k-1}}$ is the only factor in \eqref{eq:tildew} that tends to zero.
Hence the final statement of the lemma is immediate.
So assume now that $\sum_k \w_k = 1$, and $\w_k$ decreases faster than 
$e^{\beta - p S_{k-1}}$.

Taking a linear combination of two consecutive equations in 
\eqref{eq:tildew}, we obtain
\begin{equation}\label{eq:difference_tildew}
e^{pS_{k-1}} \w_k - e^{pS_k} \w_{k+1} = 
e^\beta \w_{k-1}.
\end{equation}
By setting $\w_k = e^{\beta k - p S_{k+1} + \alpha_k}$,
for some $\alpha_k\in \R$, we rewrite \eqref{eq:difference_tildew} as
$$
1-e^{\beta - pS_{k-1} + \alpha_{k+1}-\alpha_k} = e^{\alpha_{k-1}-\alpha_k}. 
$$
Abbreviating $\eps_k = \alpha_k - \alpha_{k-1}$, we have
\begin{equation*}
1-e^{\beta - pS_{k-1} - \eps_{k+1}} = e^{\eps_k}. %\label{eq:new rec}
\end{equation*}
This means that $\eps_k \to 0$ exponentially and hence 
$\alpha_k$ converges to some limit 
$\alpha_\infty = \alpha_\infty(p,t)$,
exponentially fast in $k$.  Therefore, $\w_k \to 0$ 
super-exponentially in $k$, whenever $p > 0$.
\end{proof}

\subsection{Upper bounds on $\Pconf(\phi_t)$}

We define upper bounds on $p(t)$ using a non-autonomous dynamical system.  The following lemma will be applied to this.

\begin{lemma}\label{lem:eta}
The map $\eta:r \mapsto 1-\frac{\xi}{4r}$ has
$$
\begin{cases}
\text{one fixed point } \frac12 & \text{if } \xi = 1;\\
\text{two fixed points } \FP_\pm =  \frac12(1\pm \sqrt{1-\xi}) \qquad & \text{if } \xi < 1;\\
\text{no fixed points } & \text{if } \xi > 1.
\end{cases}
$$
If $\xi < 1$, then the largest fixed point $\frac12(1+\sqrt{1-\xi})$
is attracting; if $\xi \le 0$, then the interval $[1, \infty)$ is invariant.
If $\xi > 1$, and $\delta = \sqrt{ \frac{2(\xi-1)}{3(\sqrt{\xi}+1)} }$
then it takes an orbit at least $\sqrt{\frac{3(\sqrt{\xi}+1)}{2(\xi-1)}}$
iterates to pass through the interval $[\frac{\sqrt\xi}{2}-\delta,
\frac{\sqrt\xi}{2}+\delta]$. 
\end{lemma}

\begin{proof}
The first statements follow from straightforward calculus.
For the last statement, observe that 
$\eta'(r) = 1$ for $r = \frac{\sqrt{\xi}}{2}$ and  
the vertical distance $r-\eta(r) = \sqrt{\xi}-1$.
Furthermore
$(r+\delta)-\eta(r+\delta) \le (\sqrt{\xi}-1) + \frac{2}{\sqrt{\xi}}\delta^2$.
and  $(r-\delta)-\eta(r-\delta) \le (\sqrt{\xi}-1) + \frac{2}{\sqrt{\xi}}\delta^2+O(\delta^3)$.
Hence for $\xi$ sufficiently close to $1$, we have 
$x-\eta(x) \le (\sqrt{\xi}-1)+(3/2)^2\delta^2$ for all 
$x \in [\frac{\sqrt\xi}{2}-\delta, \frac{\sqrt\xi}{2}+\delta]$,
so it takes an orbit at least
$2\delta/[(\sqrt{\xi}-1)+(3\delta/2)^2]$ iterates to pass through this interval.
This quantity is maximised for 
\begin{equation}\label{eq:delta}
\delta = \sqrt{\frac{2(\sqrt{\xi}-1)}{3}} = 
\sqrt{ \frac{2 (\xi-1)}{3(\sqrt{\xi}+1)} },
\end{equation}
in which case 
$(r+\delta)-\eta(r+\delta) \le (1+\frac{4}{3\sqrt \xi})(\sqrt{\xi}-1)$.
In this case, it takes at least $\sqrt{\frac{3(\sqrt{\xi}+1)}{2(\xi-1)}}$
 iterates to pass through the interval.
\end{proof}

\begin{lemma}\label{lem:eta2}
Let $(u_k)$ be given by 
$$
u_1 = \lambda^t \qquad \text{ and } \qquad 
u_{k+1} = \eta_k(u_k) := 1 - \frac{e^{\beta' - p S_{k-2}}}{4 u_k}.
$$
There exist constants $\tau_1 = \tau_1(\lambda), \tau_1' = \tau_1'(\lambda)$ 
(with precise values given in the proof) such that if
$$
p > \begin{cases}
\tau_1 e^{-\frac{5 \Gamma}{6 \sqrt{t_1-t}} } & \text{ if } \lambda \ge \frac12,\ t < t_1 \text{ close to } t_1,\quad \Gamma = \frac{2\log \gamma}{\sqrt{-\log [\lambda(1-\lambda)]}}\\
\tau_1' (1-t)^{ \frac{\lambda \log \gamma}{2t(1-2\lambda)} } & \text{ if } \lambda < \frac12,\ t < 1 \text{ close to } 1,
\end{cases}
$$
then $u_k \ge \frac13$ for all $k$ and $u_k \to 1$ exponentially.
\end{lemma}

\begin{proof} 
Let $\xi_k = e^{\beta'- p S_{k-2}}$. 
The dynamics of the map $\eta_k: r \mapsto r - \frac{\xi_k}{4r}$ depend
crucially on whether $\xi_k > 1$ or $\xi_k \le 1$.
These cases are roughly parallel to 
$\lambda \ge \frac12,\ t < t_1 \text{ close to } t_1$ and 
$ \lambda < \frac12,\ t < 1 \text{ close to } 1$.
However, if $pS_{k-2}$ is sufficiently large, the factor $e^{-pS_{k-2}}$ turns 
the first case into the second.

By Lemma~\ref{lem:eta},
if $\xi_k \le 1$, then  $\eta_k$ has an attracting fixed point,
tending to $1$ as $\xi_k \to 0$. 
Therefore, once $\beta' - pS_{k-2} \le 0$, and assuming that
$u_k \geq \FP_k$ where $\FP_k \le \frac12$ is the repelling fixed point of $\eta_k$,
the orbit of $u_k$ will tend
to the attracting fixed point which itself moves to $1$ at an exponential rate
as $k \to \infty$.
\iffalse
\begin{figure}
\begin{center}
\unitlength=7.5mm
\begin{picture}(20, 7)(-0.5,1)
\put(1,1.5){\resizebox{12cm}{5cm}{\includegraphics{FigLem6}}}
\end{picture}
\caption{The graphs of $\eta_1$ and $\eta_k$ for $\xi_1 \le 1$ (left)
and  $\xi_1 > 1$ (right) with the line $m$ (parallel to tangent line $m'$ to
$\eta_k$ at $\FP_- - \sqrt{\eps}$ drawn in.
The orbit  of $u_1$ under $\eta_1$ is depicted, but the actual (non-autonomous) orbit $(u_j)_{j \ge 1}$ we always have $\eta^{k-1}(u_1) \le u_k$.}
\label{fig:etak.eps}
\end{center}
\end{figure}
\fi
However, if $\xi_k > 1$, \ie $\xi_k$ is ``before'' the  saddle node bifurcation that produces the fixed point $\frac12$, then
$u_k$ will decrease and eventually become negative.
The crux of the proof is therefore to show that the
``tunnel'' between the graph of $\eta_k$ and the diagonal closes up
before the orbit $(u_k)_{k \ge 0}$ has moved through this tunnel.

We fix $\xi = e^{\beta'}$ and $\delta $ as in \eqref{eq:delta},
and we will choose $p$ so that the repelling fixed
point $\FP_k$ is to the left of the tunnel (of width $2\delta$ and centred around the point $x = \sqrt{\xi}/2$ at which $\eta'(x) = 1$), 
\ie
$$
\FP_k = \frac12\left(1-\sqrt{1-\xi e^{-pS_{k-2}}}\right) \le
\frac12 \sqrt{\xi} - \delta = \frac12 \left( \sqrt{\xi} - G \sqrt{\xi-1}  \right),
$$
where we abbreviated $G = \sqrt{\frac{8}{3(\sqrt{\xi}+1)}} < \frac{6}{5}$.
Assuming that equality holds, and solving for $e^{-pS_{k-2}}$, we obtain
$$
1- \xi e^{-pS_{k-2}} = (\sqrt\xi - 1 ) \left(\sqrt\xi -1 - 2G \sqrt{\xi-1}  + (\sqrt \xi + 1) G^2\right),
$$
which can be reduced to $pS_{k-2} = 2(1+G^2) (\sqrt\xi - 1) + o(\sqrt\xi - 1)$.

Lemma~\ref{lem:eta} states that the passage through the tunnel takes at least 
$$
k = \sqrt{\frac{3(\sqrt{\xi}+1)}{2(\xi-1)}} = \frac{2}{G \sqrt{\xi - 1}}
\le \frac{5}{6} \frac{2}{\sqrt{-\log[\lambda(1-\lambda)] (t_2-t)}}
$$
iterates. 
Note that $2(1+G^2) = 14/3 < 5$. Choose $\tau_1 = 5\gamma^2 (\sqrt{e^{\beta'}}-1)
= 5\gamma^2 (\sqrt\xi - 1)$
and  $p \ge \tau_1 e^{-\frac{5 \Gamma}{6 \sqrt{t_2-t}}}$.
Then
$$
p \ge \tau_1 e^{-\frac{5}{6} \frac{\Gamma}{\sqrt{t_2-t}} }
\ge \tau_1 e^{-k \log \gamma}
= 5 \gamma^2 (\sqrt\xi-1) \gamma^{-k} > \frac{2(1+G^2)(\sqrt\xi - 1)}{S_{k-2}}.
$$
Hence $pS_{k-2} > (1+G^2) (\sqrt \xi - 1) + o(\sqrt\xi - 1)$ and we conclude that the tunnel closes with fixed point to the left of the tunnel, before
$u_k$ passes through it.
At (or before) this iterate, $u_k$ starts to increase again, and 
eventually converge to $1$ at an exponential rate.

Now let us assume that $\xi < 1$, so there is a (left) fixed point
$\FP_- = \frac12(1-\sqrt{1-\xi})$ which for $t = 1$ coincides
with $u_1 = \lambda^t$.
For $t < 1$, we have $u_1 < \FP_-$, say $\FP_1-u_1 = \eps = \eps(t)$, and 
Taylor expansion shows that 
$$
\eps(t) = \frac12\left(1-2\lambda^t - \sqrt{1-4e^\beta}\right)
= C(1-t) + O((1-t)^2)
$$ 
for $C = \lambda \log \lambda - \lambda(1-\lambda) 
\frac{\log[\lambda(1-\lambda)]}{1-2\lambda}$.
Assume that $j$ is the first iterate such that $u_1-\sqrt{\eps} \ge u_j$.
Let $K = \eta'(u_1-\sqrt{\eps}) \le \frac{e^\beta}{ (u_1-\sqrt{\eps})^2 }
\approx  \frac{e^\beta}{u_1^2} = \frac{(1-\lambda)^t}{\lambda^t}$.
By taking a line with slope $K$ through the point $(\FP_-, \FP_-)$ to
approximate the graph of $\eta_j$ (and this line lies below the graph
of $\eta_k$ on the interval $[u_1-\sqrt{\eps}, \FP_-]$), we can estimate
$u_j \ge u_1-K^{j-1} \eps$, so $K^{j-1} \ge 1/\sqrt{\eps}$.

Due to the inequality $\log K \le t \log \left(\frac{1-\lambda}{\lambda}\right) \le
t \left(\frac{1-2\lambda}{\lambda}\right)$, and taking $\tau_1' = (\log 5) C^{ \frac{\lambda \log \gamma}{2t(1-2\lambda)} }$,
the condition
$p > \tau_1' (1-t)^{ \frac{\lambda \log \gamma}{2t(1-2\lambda)} }$ implies
$$
p > \log 5 \cdot \left(C (1-t) \right)^{ \frac{\log \gamma}{2\log K} }
\ge (\log 4) (\sqrt{\eps})^{\frac{\log \gamma}{ \log K}} \ge
\frac{\log 4}{ K^{ (j-1) \frac{\log \gamma}{\log K} } } =
\frac{\log 4}{\gamma^{j-1}} \approx \frac{\log 4}{S_{j-1}}.
$$
Let $\FP_j$ be the left fixed point of 
$\eta_j$.  Thus, given that $p \ge (\log 4)/S_{j-1}$, we find that 
$\FP_j = \frac12(1-\sqrt{1-e^{\beta'-pS_{j-1}}}) > \lambda^t/4$
for $t$ close to $1$.
Therefore $\FP_j < u_j$, and $u_j$ will converge to the attracting fixed point $\FP_+$, which itself converges exponentially to $1$.
\end{proof}

\begin{proposition}\label{prop:upperboundp0}
For the constants $\tau_1,\tau_1'$ from Lemma~\ref{lem:eta2}
we have the following upper bounds for the pressure:

$$
p(t) \le \begin{cases}
\tau_1 e^{-\frac{5}{6}\frac{\Gamma}{\sqrt{t_1-t}}} & \text{ if } \lambda \ge \frac12,\ t < t_1 \text{ close to } t_1; \\
\tau_1' (1-t)^{ \frac{\lambda \log \gamma}{2t(1-2\lambda)} } & \text{ if } \lambda < \frac12,\ t < 1 \text{ close to } 1.
\end{cases}
$$
\end{proposition}

\begin{proof}
Let $u_k = \frac{\w_{k+1}}{\w_k} e^{p S_{k-2}}$, so
$u_1 = \lambda^t e^{p(S_{-1} + S_0 - S_1)} = \lambda^t > \frac12$ (where we set $S_{-1} = 1$ by default).
From \eqref{eq:wkrecursive} we have
\begin{equation} \label{eq:uk}
u_{k+1} = 1 - \frac{e^{\beta' - p S_{k-2}}}{4 u_k}.
\end{equation}
For $p > \tau_1 e^{-\frac{5}{6}\frac{\Gamma}{\sqrt{t_1-t}}}$
or $p > \tau_1 (1-t)^{ \frac{\lambda \log \gamma}{2t(1-2\lambda)} } $
as given in Lemma~\ref{lem:eta2}, the iterates 
$u_k$ are bounded away from zero and $u_k \to 1$ exponentially.
Therefore
\begin{eqnarray*}
u_\infty := \prod_{j \ge 1} u_j &=& 
u_1 \cdot \prod_{j \ge 2}
\left( 1-\frac{e^{\beta'-pS_{j-1}}}{4u_{j-1}} \right) \\
&\ge& r_2 \cdot \prod_{j \ge 3}
\left( 1-\frac{3e^{\beta'-pS_{j-1}}}{8} \right) > 0
\end{eqnarray*}
because $\sum_{j \ge 2} 3e^{\beta'-pS_{j-1}}/8 < \infty$ and all terms $e^{\beta'-pS_{j-1}}/8$ are uniformly bounded away from 1.
Since
$$
\w_{k+1} = e^{-p(S_{k-2}+S_{k-1} + \dots + S_{-1})}
 \cdot \w_1 \cdot  \prod_{j=1}^k u_j,
$$
it follows that all $\w_k$ are positive, and as
$\w_k = e^{\beta-pS_{j-1}} \left(1-\sum_{j < k-1} \w_j \right)$,
also
$$
H_{k-1}(p,t) := \sum_{j < k-1} \w_k \le 1
$$ 
for all $k$.

To prove that $H_k(p) < 1$ for $p > \tau_1 e^{-\frac{5}{6}\frac{\Gamma}{\sqrt{t_1-t}}}$
or $p > \tau_1' (1-t)^{ \frac{\lambda \log \gamma}{2t(1-2\lambda)} } $, 
we will show that 
$\frac{\partial H_k(p)}{\partial p} < 0$ for these values of $p$.
Observe that $\partial H(p,t)/\partial p$
satisfy the recursive relation:
$$
\begin{array}{ll}
H_1 = (1-\lambda)^t e^{-p} & 
\frac{\partial H_1}{\partial p} = -(1-\lambda)^t e^{-p} < 0. \\[2mm]
H_2 = (1-\lambda)^t e^{-p} + e^\beta e^{-2p} & 
\frac{\partial H_2}{\partial p} = -(1-\lambda)^t e^{-p}
- 2e^{\beta-2p} < \frac{\partial H_1}{\partial p}. \\[2mm]
\qquad \vdots & \qquad\, \vdots \\[2mm]
H_j = H_{j-1} + e^{\beta-pS_{j-1}}(1-H_{j-2}) \quad & 
\frac{\partial H_j}{\partial p} = 
\frac{\partial H_{j-1}}{\partial p}-
e^{\beta-pS_{j-1}}\frac{\partial H_{j-2}}{\partial p} \\[2mm]
 & 
\qquad\quad  -\ S_{j-1}e^{\beta - pS_{j-1}}(1-H_{j-2}). 
\end{array}
$$
Writing $U_j := \frac{\partial H_j}{\partial p}$ and 
$q_{j+1} := U_{j-1}/U_{j-2}$, we find $q_4 = 1+2e^{\beta-2p} > 1$ and
$$
q_{j+1} = \frac{U_{j-1}}{U_{j-2}} 
=  \eta_j(q_j) := 1 - \frac{e^{\beta'-pS_{j-2}}}{ 4q_j} \left(1+S_{j-2} \frac{1-H_{j-3}}{U_{j-3}}\right) \ge 1 - \frac{e^{\beta'-pS_{j-2}}}{ 4q_{j-2}},
$$
where the final inequality relies on $U_{j-3}$ being negative.
This follows from induction, combined with Lemma~\ref{lem:eta2},
which implies that
$q_k \ge \frac13$ and $q_k \to 1$ exponentially fast,
so $\prod_{i \ge 4} q_i > 0$.
It follows that 
$$
 \frac{\partial H}{\partial p} = 
\lim_{j\to \infty} U_j = U_1 \cdot \lim_{j \to \infty} \prod_{i=4}^j q_i  < 0,
$$
and hence $H(p,t) < 1$.
\end{proof}

\begin{remark}
The techniques in this proof give no explicit formula for 
$\frac{\partial H}{\partial p}$ and $\frac{\partial H}{\partial t}$
as $t \nearrow t_1$, so they don't answer the question whether 
$\frac{dp}{dt} \to 0$ as $t \nearrow t_1$.
\end{remark}

\subsection{Existence and uniqueness of $\Pconf(\phi_t)$}

\begin{lemma}
For all $t< t_1$ there exists $p_u\ge p_\ell\ge 0$ such that $H(p_\ell,t)= 1$ and $\w_i(p_\ell)\ge 0$ for all $i$, and $H(p, t)<1$ for all $p\ge p_u$.
\label{lem:H finite}
\end{lemma}

We will show later in this section that in fact $p_u=p_\ell$ for $t$ close to $t_1$; and then 
 in Section~\ref{sec:CMS} that this is actually true for all $t$.

\begin{proof}
For any $p>0$, since $\w_i(p)\le e^{\beta-pS_{k-1}}$, we have $H(p,t)<\infty$.  
This fact also implies that $H(p,t)<1$ for all large $p$, thus proving the existence of $p_u$.  

For each $(p,t)$, define the partial sums $H_j = H_j(p,t) := \sum_{i \le j} \w_i(p)$. 
Recall from Lemma~\ref{lem:tildewp=0}
that there is some minimal $k_0 \in \N$ such that $\w_{k_0}(0)=\bar w_{k_0}^t<0$.  
By the recurrence relations defining $\w_k(0)$, this means 
that $H_{k_0-2}(0,t)>1$.
Now we prove the existence of a solution to the equation 
$H(p,t)=1$ with all $\w_j(p)\ge 0$ by continuity.
For $k \in \N$, let
$$
p_k := \inf\left\{p\ge 0:H_j(p',t)<1 \text{ for each } j \le k \text{ and } p'\ge p\right\}.
$$ 
We collect some facts: 

\begin{enumerate}[label=\textbullet,  itemsep=0.0mm, topsep=0.0mm, leftmargin=7mm]
\item $\sup_k p_k\in (0,\infty)$.  
Since $H_{k_0-2}(0,t)>1$ for some $k_0\in \N$ as shown before, 
combined with the fact that 
 $(p_k)_{k\ge 1}$ is a non-decreasing (which follows immediately from the definition of $p_j$)
gives that  $\sup_k p_k > 0$. The finiteness follows 
from the bound $\w_i(p)\le e^{-pS_{k-1}}$. 
\item If $p_k>0$ then $H_k(p_k,t)=1$.  This follows since each map $p\mapsto H_k(p,t)$ is continuous in $p$, so by the definition of $p_k$ as an infimum, there must exist a minimal $j \le k$ such that $H_j(p_k, t)=1$.  
But our recurrence relation 
\eqref{eq:difference_tildew} implies that $H_{j+1}(p_k,t)= H_j(p_k,t) + e^{\beta-p_kS_k}(1-H_{j-1}(p_k,t))>H_j(p_k,t)=1$.  
This must also hold for all $p$ sufficiently close to $p_k$,  
so if $j<k$ then this contradicts the definition of $p_k$.
\item $\w_j(p')\ge 0$ for all $j\le k$ and $p'\ge p_k$.  
If this fails, take the minimum such $k$ and note that \eqref{eq:difference_tildew} implies that 
$H_{j-2}(p', t)>1$, a contradiction.
\end{enumerate}

Now define $p_\infty:=\sup_k p_k$.   It follows immediately from this definition that for any $j \in \N$, $H_j(p_\infty, t)<1$ so $H(p_\infty,t)\le 1$.  
Note that this also implies that $\w_j(p_\infty)\ge 0$ for all $j\in \N$. 

To show that $H(p_\infty, t)=1$, notice that for $p>0$ and any $j\in \N$, 
$$
H(p, t) = H_j(p,t)+ \sum_{k>j}\w_k(p)\ge H_j(p,t)-e^\beta\sum_{k>j}e^{-pS_k}.
$$
So defining $j_0\in\N$ such that $p_{j_0}>0$, let $s(j):=e^\beta \sum_{k>j}e^{-p_{j_0}S_k}$.  Then for $p_j\ge p_{j_0}$,
$$H(p_j,t)\ge H_j(p_j, t)-s(j)=1-s(j).$$
So since $s(j)\to 0$ as $j\to \infty$, we have $H(p_j,t)\to 1$ as $j\to \infty$.  Therefore, the continuity of $p\mapsto H(p,t)$ on the domain where the sums are bounded implies that $H(p_\infty,t)= 1$.
\end{proof}

\begin{proposition}\label{prop:Umonotone}
There is at most one solution $p = p(t)$ to $H(p,t) = 1$ with all
$\w_k > 0$. Moreover,
$\frac{\partial H}{\partial p} < 0$, $\frac{\partial H}{\partial t} < 0$,
and the map $t \mapsto p(t)$ is analytic with  $\frac{dp}{dt} < 0$
on $(t_1-\eps, t_1)$.
\end{proposition}

\begin{proof} The proof uses many of the ideas of the proof of Proposition~\ref{prop:upperboundp0}.
The previous proof shows that positivity of all $\w_k$ is equivalent to
positivity of the numbers $u_k = \frac{\w_{k+1}}{\w_k} e^{pS_{k-1}}$ from \eqref{eq:uk}. 
Therefore, if $p = p(t)$ is a solution to the problem
$\w_k > 0$ and $H(p,t) = 1$, then
the corresponding sequence $(u_k)_k$ is positive.
Positivity of $\frac{\partial H(p,t)}{\partial p}$ is equivalent to
positivity of an orbit $(v_k)_k$ for a slightly different but larger map,
and with an initial value $v_4 > 1 \ge u_1$.
Therefore, as $(u_k)_k$ is positive, so is $(v_k)_k$, and
$0 < \prod_{k \ge 1} u_k \le \prod_k v_{k \ge 4} = v_\infty$, whence 
$\frac{\partial H(p,t)}{\partial p} = v_\infty \cdot \frac{\partial H_1(p,t)}{\partial p} < 0$.
This shows that there can be at most one solution to $H(p,t) = 1$.

We can use the same technique to estimate
$\frac{\partial H(p,t)}{\partial t}$ for $t < t_1$:
$$
\begin{array}{ll}
H_1 = (1-\lambda)^t e^{-p} & 
\frac{\partial H_1}{\partial t} = \log(1-\lambda)(1-\lambda)^t e^{-p} < 0. \\[2mm]
H_2 = (1-\lambda)^t e^{-p} + e^\beta e^{-2p} & 
\frac{\partial H_2}{\partial t} = \log(1-\lambda)(1-\lambda)^t e^{-p} \\[2mm]
& \qquad \qquad
+\ \log[\lambda(1-\lambda)] e^{\beta-2p} < \frac{\partial H_1}{\partial t}. \\[2mm]
\qquad \vdots & \qquad \vdots \\[2mm]
H_j = H_{j-1} + e^{\beta-pS_{j-1}}(1-H_{j-2}) \quad & 
\frac{\partial H_j}{\partial t} = 
\frac{\partial H_{j-1}}{\partial t} - e^{\beta-pS_{j-1}}\frac{\partial H_{j-2}}{\partial t} \\[2mm]
 & 
\qquad\quad  +\ \log[\lambda(1-\lambda)]e^{\beta - pS_{j-1}}(1-H_{j-2}). 
\end{array}
$$
If we now write $U_j = \frac{\partial H_j}{\partial t}$ and 
$q_{j+1} := U_{j-1}/U_{j-2}$, we find 
$q_4 = 1+\frac{\log [\lambda(1-\lambda)] e^{\beta-p}}{\log(1-\lambda) (1-\lambda)^t} > 1$ and
$$
q_{j+1} = \frac{U_{j-1}}{U_{j-2}} 
=  1 - \frac{e^{\beta'-pS_{j-2}}}{ 4q_j} \left(1-\log[\lambda(1-\lambda)] \frac{1-H_{j-3}}{U_{j-3}}\right) \ge 1 - \frac{e^{\beta'-pS_{j-2}}}{ 4q_{j-2}},
$$
where the final inequality relies on $U_{j-3}$ being negative.
The same argument shows that $\frac{\partial H}{\partial t} < 0$ as well.
Furthermore, since $H(p,t)$ is analytic in both $p$ and $t$, the Implicit Function Theorem  implies that
$t \mapsto p(t)$ is analytic on $(t_1-\eps, t_1)$ and $\frac{dp}{dt} < 0$.
\end{proof}

\section{Invariant measures}\label{sec:invariant}
Now we look at the invariant measure $\mu_{t,p} \ll m_{t,p}$ for $t < t_1$.

\begin{theorem}\label{thm:existence_invariant_measures}
Suppose $t < t_1$ and $p>0$ satisfies $H(p,t)=1$ with all summands non-negative.
Then we have the following:
\begin{enumerate}[label=({\alph*}),  itemsep=0.0mm, topsep=0.0mm, leftmargin=7mm]
\item There is an $F_\lambda$-invariant measure $\mu_t=\mu_{t,p}\ll m_{t,p}$;
\item The Radon-Nikodym derivative 
$\frac{d\mu_t}{dm_t}$ is bounded and bounded away from zero;
\item  $\mu_t$ projects to an $f_\lambda$-invariant probability measure $\nu_t \ll n_t$.
\end{enumerate}
\end{theorem}

\begin{proof}
The solution $\tilde {\underline w}^t$ to \eqref{eq:tildew} and
$H(p,t) = 1$ gives rise to a probability  transition matrix
\begin{equation*}%\label{eq:matrixG}
G^t = \left( \begin{array}{cccccc}
\w_1    & \w_2 & \w_3 & \w_4 & \hdots & \hdots   \\[2mm]
\w_1   & \w_2 & \w_3 & \w_4 & \hdots & \hdots  \\[2mm]
0      & \frac{\w_2}{\sum_{i \ge 2}\w_i}       & \frac{\w_3}{\sum_{i \ge 2}\w_i}   & \frac{\w_4}{\sum_{i \ge 2}\w_i}  & \hdots &  \\[2mm]
0   & 0   & \frac{\w_3}{\sum_{i \ge 3}\w_i}       & \frac{\w_4}{\sum_{i \ge 3}\w_i}   &   & \\
0   & 0   & 0 & \ddots   & \ddots   &  \\
\vdots   &    &   & \ddots   & \ddots   & 
\end{array} \right).
\end{equation*}
The left eigenvector $\tilde{\underline{v}}^t
= (\tilde v_1^t, \tilde v_2^t, \dots )$ for eigenvalue $1$ represents the invariant measure: $\mu_{t,p}(W_k \cup \hat W_k) = \tilde v_k^t$.
To find it, we start with $v^{(0)} := (1,0,0,\dots)$ and iterate
$v^{(n)} = v^{(n-1)} G^t$.
Since $G^t$ is a stochastic matrix (\ie nonnegative and with row-sums are $1$),
each $v^{(n)}$ is non-negative and has $\| v^{(n)}\|_1 =1$ as well.
We prove by induction in $n$ that $v^{(n)}_j$ decreases super-exponentially in $j$.
We will show that there is $K \in \N$ such that for all $n \ge 0$,
\begin{equation}\label{eq:super-polynomial}
v^{(n)}_k \le  \frac{\w_{k-1}}{\w_{k-2}}\ \text{ for all } k \ge K.
\end{equation}
Since $\w_k = e^{\beta k - \gamma^2p S_{k+1} + \alpha_k}$ 
decrease super-exponentially in $k$
as described in \eqref{eq:wkalpha}
we can find $K$ such that
$\frac{\w_k}{\w_{k-1}} \le \frac12  \frac{\w_{k-1}}{\w_{k-2}}$  for $k\ge K$.
Clearly \eqref{eq:super-polynomial} holds for $v^{(0)}$. For the inductive step,
assume \eqref{eq:super-polynomial} holds for $n-1$.
Then for $k \ge K$ arbitrary, 
\begin{eqnarray}\label{eq:vkn1}
v_k^{(n)} &=& \w_k \left( v_1^{(n-1)} + v_2^{(n-1)} 
+ \sum_{j=3}^{k+1} \frac{v_j^{(n-1)}}{ \sum_{i \ge j-1} \w_i} \right)\\
&=& \w_k \left( v_1^{(n-1)} + v_2^{(n-1)} 
+ \sum_{j=3}^k \frac{v_j^{(n-1)}}{ \sum_{i \ge j-1} \w_i} \right)
+ v_{k+1}^{(n-1)} \left( 1- \frac{ \sum_{i \ge k+1} \w_i}{ \sum_{i \ge k} \w_i} \right) \nonumber \\
&\le& \left(v_1^{(n-1)} + \dots + v_{k-1}^{(n-1)}
+  v_k^{(n-1)} \right) \frac{\w_k}{\w_{k-1}} + v_{k+1}^{(n-1)}\nonumber \\
&\le& \frac{\w_k}{\w_{k-1}} +  \frac{\w_k}{\w_{k-1}} \le
 \frac{\w_{k-1}}{\w_{k-2}},\nonumber
\end{eqnarray}
where in the last line we used that $\| v^{(n-1)} \| = 1$ as well the choice of $K$.
This shows that although the unit ball in $l^1$ is not compact, 
the sequence $( v^{(n)})_{n \ge 0}$ is tight, and hence must have
a convergent subsequence.
Since $G^t$ is clearly an irreducible aperiodic matrix,
$(v^{(n)})_{n \ge 0}$ converges; let $\tilde v^t$ be the limit.
Then $\tilde v^t$ is positive and $\| \tilde v^t\|_1 = 1$.

The measure $\mu_t$ defined by the piecewise constant Radon-Nikodym derivative 
$h_k := h|_{W_k \cup \hat W_k} = \frac{\mu_t(W_k \cup \hat W_k)}{m_t(W_k \cup \hat W_k)} = \frac{\tilde v^t_k}{\w_k}$
is now easily seen to be invariant. By taking the limit $n \to \infty$ in
\eqref{eq:vkn1}, we find 
$$
\tilde v_k^t = \w_k \left( \tilde v_1^t + \tilde v_2^t 
+ \sum_{j=3}^{k+1} \frac{\tilde v_j^t}{ \sum_{i \ge j-1} \w_i} \right)\\
= \frac{\w_k}{\w_{k-1}} \tilde v_{k-1}^t
+ \frac{\w_k}{ \sum_{i \ge j-1} \w_i} \tilde v_{k+1}^t,
$$
and dividing this by $\w_k$ shows that $(h_k)_{k \in \N}$ is increasing,
and hence bounded away from $0$.
Now for the upper bound, taking the limit  $n \to \infty$ in \eqref{eq:super-polynomial} 
shows that  $\tilde v_k^t \to 0$ super-exponentially fast.
Take $K \in \N$ such that
$$
\sum_{k \ge K} \frac{\w_k}{\w_{k-1}} < \frac14
$$
Then by \eqref{eq:vkn1}:
$$
h_k = \frac{\tilde v_k^t}{\w_k} \le \underbrace{\tilde v_1^t +  \tilde v_2^t +
\sum_{j=1}^K \frac{\tilde v_j^t}{\w_{j-1}} }_C + 
 \sum_{j=K+1}^{k-1} \frac{\w_j}{\w_{j-1}} h_j
+
\frac{\w_k}{\w_{k-1}} h_k
+ \frac{\tilde v_{k+1}^t}{\tilde v_k^t} h_k.
$$
This gives 
\begin{equation}
h_k \le \frac{C + (\sup_{j < k} h_j) \cdot \sum_{j=K+1}^{k-1} \frac{\w_j}{\w_{j-1}}}{1 - \frac{\w_k}{\w_{k-1}} - \frac{\tilde v_{k+1}^t}{\tilde v_k^t}}.
\label{eq:h_k}
\end{equation}
If $\frac{\tilde v_{k+1}^t}{\tilde v_k^t} \le \frac14$, then so long as $k$ is sufficiently large,  \eqref{eq:h_k} yields $h_k \le 2C + \frac12 \sup_{j < k} h_j$.  Since $h_k$ is an increasing sequence, 
we conclude that $h_k\le 4C$ and moreover, $h_j\le 4C$ for all $j\le k$.  The fact that $\tilde v_k^t\to 0$ super-exponentially implies that there are infinitely many $k$ satisfying $\frac{\tilde v_{k+1}^t}{\tilde v_k^t} \le \frac14$, so $h_j\le 4C$ for all $j\in \N$, concluding the upper bound.

Since $\mu_t(W_i \cup \hat W_i) = \tilde v_k^t$ decreases super-exponentially,
 $\Lambda := \sum_j S_{j-1} \mu_t(W_j \cup \hat W_j) < \infty$ for 
$t < t_1$, so by Lemma~\ref{lem:conformal-induced}, $\mu_t$ pulls back to an $f_\lambda$-invariant probability
measure $\nu_t \ll n_t$.
\end{proof}

\section{Thermodynamic formalism for countable Markov shifts}
\label{sec:CMS}

\subsection{Countable Markov shifts}
In previous sections we have computed quantities such as pressure rather directly, which gives a fuller understanding of the underlying properties of our class of dynamical systems.
In this section we use the theory of countable Markov shifts, as developed by Sarig, to prove stronger results more indirectly.  In particular, we can obtain information about the pressure and equilibrium states for $\phi_t$ for \emph{all} $t\in \R$.

Let $\sigma : \Sigma \to \Sigma$ be a one-sided Markov shift
with a countable alphabet $\N$. That is, there exists a matrix
$(t_{ij})_{\N \times \N}$ of zeros and ones (with no row and no column
made entirely of zeros) such that
\[
\Sigma=\{ x\in \N^{\N_0} : t_{x_{i} x_{i+1}}=1 \ \text{for every $i
\in \N_0$}\},
\]
and the shift map is defined by $\sigma(x_0x_1 \cdots)=(x_1 x_2
\cdots)$. We say that $(\Sigma,\sigma)$ is a \emph{countable
Markov shift}.
We equip $\Sigma$ with the topology generated by the cylinder sets
$$
[e_0 \cdots e_{n-1}] = \{x \in \Sigma : x_j = e_j \text{ for } 0 \le j < n \}.
$$

Given a function
$\phi\colon \Sigma \to\R$, for each $n \ge 1$ we define the \emph{variation}
on $n$-cylinders
\[
V_{n}(\phi) = \sup \left\{|\phi(x)-\phi(y)| : x,y \in \Sigma,\
x_{i}=y_{i} \text{ for } 0 \le i < n \right\}.
\]
We say that $\phi$ has \emph{summable variations} if
$\sum_{n=2}^{\infty} V_n(\phi)<\infty$; clearly summability implies continuity of $\phi$.
In what follows we assume
$(\Sigma, \sigma)$ to be topologically mixing (see \cite[Section 2]{Sartherm} for a precise definition).

Based on work of Gurevich \cite{Gutopent, Gushiftent}, Sarig \cite{Sartherm} introduced a notion of pressure for countable Markov shifts which does not depend upon the metric of the space and which satisfies a Variational Principle.
Let $(\Sigma, \sigma)$ be a topologically mixing countable Markov shift, fix a symbol $e_0$ in the alphabet $\N$ and
let $\phi \colon \Sigma \to \R$ be a potential of summable variations.  
We let the {\em local partition function at $[e_0]$} be 
\begin{equation*}
Z_n(\phi, [e_0]):=\sum_{x:\sigma^{n}x=x} e^{S_n\phi(x)} \chi_{[e_{0}]}(x)
%\label{eq:Zn}
\end{equation*}
and
$$
Z_n^*(\phi, [e_0]) := \sum_{\stackrel{x:\sigma^{n}x=x,}{x:\sigma^{k}x \notin [e_{0}] \ \mbox{\tiny for}\ 0< k < n}} \hspace{-10mm} e^{S_n\phi(x)} \chi_{[e_{0}]}(x),
$$
where $\chi_{[e_{0}]}$ is the characteristic function of the
$1$-cylinder $[e_{0}] \subset \Sigma$,
and $S_n\phi(x)$ is $\phi(x) + \dots + \phi \circ \sigma^{n-1}(x)$.
The so-called \emph{Gurevich pressure} of $\phi$ is defined
by
the exponential growth rate
\[
 P_G(\phi) := \lim_{n \to
\infty} \frac{1}{n} \log Z_n(\phi, [e_0]).
\]
Since $\sigma$ is topologically mixing, one can show that $P_G(\phi)$ does not depend on $e_0$.  If $(\Sigma, \sigma)$ is the full-shift on a countable alphabet then the Gurevich pressure coincides with the notion of pressure introduced by Mauldin \& Urba\'nski \cite{MUifs}.

The following can be shown using the proof of \cite[Theorem 3]{Sartherm}.

\begin{proposition}[Variational Principle] \label{prop:VarPri}
If $(\Sigma, \sigma)$ is topologically mixing, $\phi: \Sigma \to \mathbb{R}$ has summable variations and $\phi<\infty$, then
\begin{equation*}
P_G(\phi)= P(\phi). 
\end{equation*}
\end{proposition}

\begin{definition}\label{def:recurrence}
The potential $\phi$ is said to be {\em recurrent} if \footnote{The convergence of this series is independent of the cylinder set $[e_0]$,
so we suppress it in the notation.}
\begin{equation*}%\label{eq:recurrence}
\sum_n e^{-nP_G(\phi)} Z_n(\phi) = \infty.
\end{equation*}
Otherwise $\phi$ is \emph{transient}.
Moreover, $\phi$ is called \emph{positive recurrent} if it is
recurrent and 
$$
\sum_n ne^{-nP_G(\phi)}Z^*_n(\phi) < \infty.
$$
If $\phi$ is recurrent but not positive recurrent, then it is called 
 \emph{null recurrent}.
\end{definition}

We use the standard {\em transfer operator}
$(L_\phi v)(x) = \sum_{\sigma y =x} e^{\phi(y)} v(y)$,
with dual operator $L^*_\phi$.
Notice that a measure $m$ is $\phi$-conformal if and only if
$L^*_\phi m = m$.

The following theorem is  \cite[Theorem 1]{Sarnull}.  Note that the next two theorems were originally proved under stronger regularity conditions
(\ie weak H\"olderness) on the potential, but subsequently it was found that these could be relaxed, see for example \cite{Sar_TDF_notes} Chapters 3 and 4.

\begin{theorem}
Suppose that $(\Sigma, \sigma)$ is topologically mixing, $\phi: \Sigma \to \mathbb{R}$ has summable variations and $P_G(\phi)<\infty$.  Then $\phi$ is recurrent if and only if there exists $\lambda>0$ and a conservative sigma-finite measure $m_\phi$ finite and positive on cylinders, and a positive continuous function $h_\phi$ such that $L_\phi^*m_\phi=\lambda m_\phi$ and $L_\phi h_\phi=\lambda h_\phi$.  In this case $\lambda=e^{P_G(\phi)}$.  Moreover,
\begin{enumerate}
\item if $\phi$ is positive recurrent then $\int h_\phi~dm_\phi<\infty$;
\item if $\phi$ is null recurrent then $\int h_\phi~dm_\phi=\infty$.
\end{enumerate}
\label{thm:RPF}
\end{theorem}

Moreover the next theorem follows by \cite[Corollary 2]{Sartherm}:

\begin{theorem}\label{thm:eq}
Suppose that $(\Sigma, \sigma)$ is topologically mixing and $\phi: \Sigma \to \mathbb{R}$ has summable variations and is positive recurrent.  Then for the measure $d\mu=h_\phi dm_\phi$ given by Theorem~\ref{thm:RPF}, if $-\int\phi~d\mu < \infty$, then $\mu$ is the unique equilibrium state for $\phi$.
\end{theorem}

We are now ready to apply this theory to our class of dynamical systems. The following proposition  contains the main ideas for the proof of Theorem~\ref{mainthm:thermo_Fibo_global}, but we state and prove it separately to highlight the connection with the results in Section~\ref{sec:conformal_for_f}.

\begin{proposition}
For each $\lambda\in (0,1)$ and any $t\le t_1$,
\begin{enumerate}[label=({\alph*}),  itemsep=0.0mm, topsep=0.0mm, leftmargin=7mm]
\item there is a unique $p$ such that $H(p,t)=1$ with all summands non-negative;
\item this $p$ is the unique value such that there is a $(\Phi_t-p)$-conformal measure.
\end{enumerate}
\label{prop:pconf p}
\end{proposition}

\begin{proof}
We first prove the proposition for the case $t<t_1$, in which case, 
any $p$ satisfying $H(p, t)=1$ with all summands non-negative, must be strictly positive.
The existence of such a $p$ follows by Lemma~\ref{lem:H finite}.  By Theorems~\ref{thm:existence_invariant_measures} and \ref{thm:RPF}, for $p$ as in (a) of the proposition, the potential $\Phi_t-\tau p$  is (positive) recurrent.  Theorems~\ref{thm:existence_invariant_measures}  also implies that  $P_G(\Phi_t-\tau p)=0$.  Since $\tau\ge 1$, for $\eps>0$ we always have $P_G(\Phi_t-\tau (p-\eps))\ge P_G(\Phi_t-\tau p)+\eps$: this means that any such $p$ is unique.  To summarise, there is one and only one $p$ such that $H(p,t)=1$ with all non-negative summands and for this $p$, we have $P_G(\Phi_t-\tau p)=0$.  It is easy to see that such a $p$ yields a $(\Phi_t-\tau p)$-conformal measure.

For the case $t=t_1$, by \cite[Theorem B]{BT3}, $P_G(\Phi_t)=0$.  Theorem B of that paper guarantees that $p=0$ is a solution to $H(p,t)=1$ with all summands positive.  The above argument also shows in this case that if there is a solution $p>0$ to $H(p,t)=1$ with all summands non-negative, then    $P_G(\Phi_t-\tau p)=0$ and again this can only occur if $p=0$.   To show that there is no negative solution, observe
$$w_j^t=e^\beta e^{-pS_{j-1}}\left(1-\sum_{k<j-1} w_k^t\right)> e^\beta e^{-pS_{j-1}}w_j.$$ Therefore we must have $p=0$ as the only solution to $H(p,t)=1$ with all summands positive.
\end{proof}

\subsection{Proof of Theorem~\ref{mainthm:thermo_Fibo_global}}

\begin{proof}[Proof of Theorem~\ref{mainthm:thermo_Fibo_global}] 
We prove parts (a) and (b) simultaneously.
First suppose that $t<t_1$.  As in the proof of Proposition~\ref{prop:pconf p}, $\Phi_t-\tau \Pconf(\phi_t)$ is positive recurrent.
By Theorem~\ref{thm:eq}, $\mu_t$ from Theorem~\ref{thm:existence_invariant_measures} 
is an equilibrium state for $\Phi_t-\tau \Pconf(\phi_t)$ and hence satisfies 
$$
h(\mu_t)+\int (\Phi_t-\tau \Pconf(\phi_t)) d\mu_t=0.
$$
Thus the Abramov formula implies that the projected measure $\nu_t$ has $h(\nu_t)+\int \phi_t~d\nu_t=\Pconf(\phi_t)$, so  $P(\phi_t)\ge \Pconf(\phi_t)$.  If  $P(\phi_t)> \Pconf(\phi_t)$ then there exists a measure $\nu$ (with positive entropy) for which $h(\nu)+\int\phi_t-\Pconf(\phi_t)~d\nu>0$.  Since any such measure must lift to $(Y, F_\lambda)$, the Abramov formula and Proposition~\ref{prop:VarPri} lead to a contradiction.  Hence $P(\phi_t)=\Pconf(\phi_t)$. This also implies that $\nu_t$ is the unique equilibrium state for $\phi_t$. 

For the case $t=t_1$, Proposition~\ref{prop:pconf p} implies that $\Pconf(\phi_t)=0$.   This is clearly the same as $P(\phi_t)$, 
as follows continuity of the pressure. 
The existence/absence of an equilibrium state here follows as in Theorem~\ref{mainthm:Fibo}.

Now let $t > t_1$.
For each $\lambda\in (0,1)$, \cite[Theorem A]{BT3} implies that the 
$\Phi_t$-conformal measure $m_t$, if it exists, is dissipative.  Hence no finite $\mu_t \ll m_t$ exists.
However, just as for smooth Fibonacci maps, $\omega(c)$ supports a unique probability measure $\nu_\omega$,
which has zero entropy.
For each $x \in \omega(c)$ not eventually mapping to $c$, $Df^n_\lambda(x)$ exists for all $n$.
Moreover, $F_\lambda^k(x) \to c$ so that if $n_k \in \N$ is such that $F^k_\lambda(x) = f^{n_k}(x)$,
then $k/n_k \to 0$. The Lyapunov exponent of $x$ under $F_\lambda$ is $-\log[\lambda(1-\lambda)]$,
hence by part c) of Lemma~\ref{lem:conformal-induced}, the Lyapunov exponent of $x$ 
under $f_\lambda$ is $\lim_{k\to\infty} - \frac{k}{n_k} \log[\lambda(1-\lambda)] = 0$.
Therefore $\nu_\omega$ is an equilibrium state in this case
(and in fact also for $t = t_1$).
This concludes the proof of (a) and (b).

Now for part c), Lemma~\ref{lem:tildewp=0} implies that 
$\Pconf(\phi_t) > 0$ for $t < t_1$.
Now if $t > t_1$, then $p = 0$
still gives a conformal measure, see \eqref{eq:mt}.
This is the smallest (and only) value of $p$ to do so, because if
$p < 0$, then $H(p,t)$ no longer converges.
Indeed, by taking the linear combinations in \eqref{eq:tildew},
we get
$$
\w_{j+1} = e^{-p S_j} \left( \w_j e^{p S_{j-1}}
 - e^\beta \w_{j-1} \right).
$$
If $p < 0$, we can no longer assert that $\w_j$ is decreasing
in $j$, but if $H(p,t)$ converges, then there must be
(infinitely many) $j$s such that $\w_j \le \w_{j-1}$.
If also $j$ is so large that $e^{-pS_{j-1}} > [\lambda(1-\lambda)]^{-t}$, then
the equation gives that  $\w_{j+1} < 0$, which is not allowed.
(The only other way of creating a conformal measure for $f$, is by
putting Dirac masses on the critical point and its backward orbit.
Since $f'(c) = 0$, this enforces no mass on the forward
critical orbit. But $f'$ is not defined at $f^{-1}(c) = \{z_0, \hat z_0\}$,
so this gives no solution.)
Therefore $\Pconf(\phi_t) = 0$ for $t \ge t_1$.

Now we turn to analyticity.
As in for example \cite{ITeqnat}, the existence of a unique equilibrium state of positive entropy implies that $p(t):t\mapsto P(\phi_t)$ is $C^1$.
(We can also use the fact that $p'(t)=-\int\log|Df_\lambda|~d\nu_t$, which is 
easily shown to be continuous in $t$.)  
It is easy to see that $Dp(t)<0$ for $t<t_1$.  
Therefore we have, as in Proposition~\ref{prop:Umonotone}, that $p(t)$ is real analytic on $(-\infty, t_1)$.
\end{proof}

\subsection{Proof of Theorem~\ref{mainthm:thermo_Fibo}}

The following proposition, which should be compared to \cite[Proposition 1.2]{ITeqnat}, will tell us the shape of the pressure function at $t_1$. 
This also gives part (d) of Theorem~\ref{mainthm:Fibo}.

\begin{proposition}
The following are equivalent.
\begin{enumerate}[label=({\alph*}),  itemsep=0.0mm, topsep=0.0mm, leftmargin=7mm]
\item The left derivative $D_-p(t_1)<0$;
\item There exists $K>0$,  $\delta>0$ so that for all $t\in (t_1-\delta, t_1)$ there is an equilibrium state $\nu_{t}$ for $-t\log|Df|$ and for the induced version $\mu_{t}$, 
$$\int \tau~d\mu_t=\sum_kS_{k-1}\mu_t(W_k)\le K.$$
\end{enumerate}
Indeed, when the above holds, there is an equilibrium state $\nu_{t_1}$ for $\phi_{t_1}$ and $\int \tau~d\mu_{t_1}\le K$.
\label{prop:deriv t_1}
\end{proposition}

\begin{proof}[Proof of Proposition~\ref{prop:deriv t_1}]
First assume that $K < \infty$ as in item (b) exists.
Since $p'(t)=-\int\log|f'_\lambda|~d\nu_t$, the Abramov formula implies
$$
\int\log|f'_\lambda|~d\nu_t=\frac{\int\log|F'_\lambda|~d\mu_t}{\int \tau~d\mu_t}\ge -\frac{\log\lambda}K,
$$
uniformly in $t$, \ie $D_-p(t_1)\le\frac{\log\lambda}K<0$.

Now let us suppose that $D_-p(t_1)<0$.  As in \cite[Lemma 4.2]{ITeqnat}, there exists $\eta>0$ such that any measure $\nu\in \M$ with $h(\nu)-t\lambda(\nu)$ sufficiently close to $p(t)$ has $h(\nu)\ge \eta$.  Suppose that $(\nu_n)_n$ is a sequence of measures such that  $h(\nu_n)-t\lambda(\nu_n)\to p(t)$. For each $n$, we denote the induced version of $\nu_n$ by $\mu_{n}$.  Now applying the Abramov formula and since $\htop(F_\lambda)=\log 4$, we obtain for all large $n$,
$$\eta\le h(\nu_n)=\frac{h(\mu_{n})}{\int \tau~d\mu_{ n}}\le \frac{\log 4}{\int \tau~d\mu_{ n}},$$
so $\int \tau~d\mu_{ n}\le (\log 4)/\eta$.

Since $\int \tau~d\mu_{ n}\le (\log 4)/\eta$ for all large $n$, for any $\eta'>0$, there must be some $N\in \N$ such that $\mu_{n}\left(\cup_{k=1}^NW_k\right)>1-\eta'$ for all large $n$.   Notice that the choice of $(\nu_n)_n$, the Abramov formula and the uniform bound on the integral of inducing times implies that
$$
h(\mu_n)-\int (\Phi_t-\tau p(t))~d\mu_n\to 0 \text{ as } n\to \infty.
$$
The proof now concludes by a tightness argument.  Let $\mu_\infty$ be a vague limit of $(\mu_n)_n$, see for example \cite[Section 28]{Bill_book}.  This measure is non-zero since $\mu_{n}\left(\cup_{k=1}^NW_k\right)>1-\eta'$ for all $n\in \N$.  We may assume that it is a probability measure.  The Monotone Convergence Theorem implies that $\int\tau~d\mu_\infty\le (\log 4)/\eta$.  Moreover, the continuity of $\Phi_t$ and the upper semi-continuity of $-\tau$  implies that $\mu_\infty$ is an equilibrium state for $\Phi_t-\tau p(t)$.  The fact that the integral of the inducing time is finite implies that we can project $\mu_\infty$ to an equilibrium state $\nu_t$ for $\phi_t$, as required.
\end{proof}

\begin{proof}[Proof of Theorem~\ref{mainthm:thermo_Fibo}]
The lower and upper bounds for the pressure on a left neighbourhood of 
$t_1$ stated in (a) and (b) follow from Proposition~\ref{prop:lowerboundp0}  (with in one case $\tau_0 \tilde C$ renamed to $\tau_0'$) and 
 Proposition~\ref{prop:upperboundp0} respectively.
Finally, part (c) follows from Proposition~\ref{prop:deriv t_1}.
\end{proof}

\subsection{Recurrence and transience}
\label{ssec:recurrence}

We finish the paper with a brief discussion of recurrence/transience in the context of our examples using the definitions given above.   Since we can view $(Y, F_\lambda)$ as a countable Markov shift, by Theorem~\ref{thm:RPF}, Proposition~\ref{prop:pconf p} and 
Theorem~\ref{thm:existence_invariant_measures} we have the following results for the system $(Y, F_\lambda, \Phi_t-\tau p)$: note that the precise behaviour at $t=t_1$ is governed by the case $p=0$ which is discussed in Section~\ref{sec:example}, see also \cite{BT3}:

\begin{enumerate}[label=$\bullet$,  itemsep=0.0mm, topsep=0.0mm, leftmargin=7mm]
\item If $\lambda\in (1/2,1)$ then $(Y, F_\lambda, \Phi_t-\tau p)$ is recurrent iff $t<t_1<1$ and $p=\Pconf(\phi_t)$.  Whenever the system is recurrent, it is positive recurrent .
\item If $\lambda\in (0, 1/2)$ then $(Y, F_\lambda, \Phi_t-\tau p)$ is recurrent iff $t\le t_1=1$  and $p=\Pconf(\phi_t)$.   Whenever the system is recurrent, it is positive recurrent.
\item  If $\lambda=1/2$ then $(Y, F_\lambda, \Phi_t-\tau p)$ is recurrent iff $t\le t_1=1$ and $p=\Pconf(\phi_t)$.  It is null recurrent for $t=1$ and positive recurrent if $t<1$.
\end{enumerate}

For the original system, the Markov shift model is less  easy to handle, so we prefer 
an alternative definition of recurrence.  In \cite{ITtransient} a system $(X, f, \phi)$ was called recurrent whenever there was a conservative $\phi$-conformal measure $m_\phi$ and transient otherwise.  A recurrent system was defined as being positive recurrent if there was an $f$-invariant probability measure $\mu_\phi\ll m_\phi$, and null-recurrent otherwise.  With this in mind, the results of this paper allow use to state:

\begin{enumerate}[label=$\bullet$,  itemsep=0.0mm, topsep=0.0mm, leftmargin=7mm]
\item If $\lambda\in (1/2,1)$ then $(I, f_\lambda, \phi_t- p)$ is recurrent iff $t<t_1<1$ and $p=\Pconf(\phi_t)$.  Whenever the system is recurrent, it is positive recurrent.
\item If $\lambda\in (0, 1/2]$ then  $(I, f_\lambda, \phi_t- p)$ is recurrent iff $t\le t_1=1$  and $p=\Pconf(\phi_t)$.   When the system is recurrent and $p=\Pconf(\phi_t)$, it is positive recurrent iff $\lambda\in (0,\frac2{3+\sqrt 5})$.
\end{enumerate}

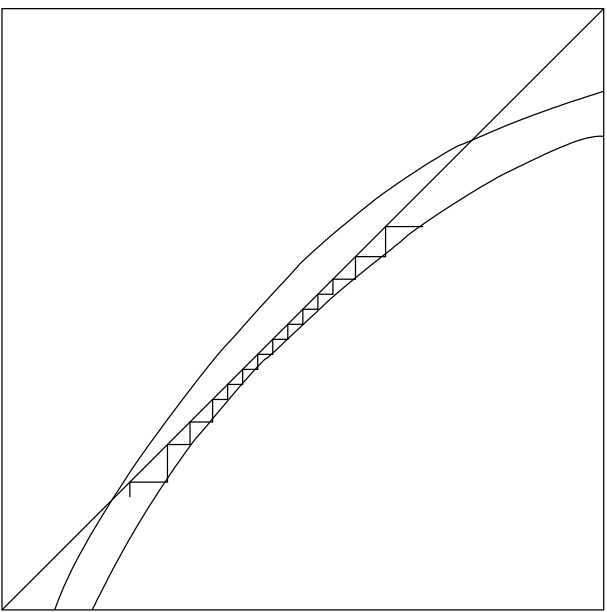
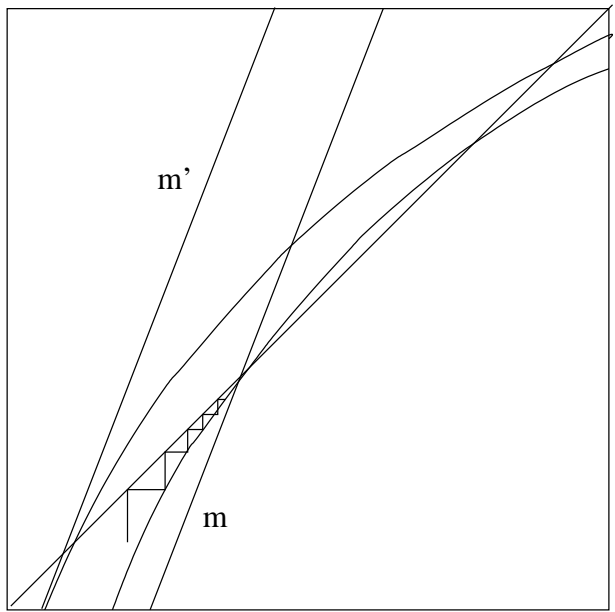
\end{document}